%% file: CanonBorromean.tex
\documentclass[12pt]{amsart}

\usepackage{amsthm}
\usepackage{amsmath}
\usepackage{amssymb}
\usepackage{enumerate}
\usepackage{graphicx}
\usepackage[hidelinks,pdftex]{hyperref}
\usepackage{booktabs}
\usepackage{color}
\usepackage[dvipsnames]{xcolor}
\usepackage{import}
\usepackage{tikz-cd}
\usepackage{microtype}

\usepackage[margin=3.2cm]{geometry}

\AtBeginDocument{%
   \def\MR#1{}
}

\numberwithin{equation}{section}
\numberwithin{figure}{section}

\theoremstyle{plain}
\newtheorem{theorem}[equation]{Theorem}

\newtheorem{lemma}[equation]{Lemma}

\newtheorem{proposition}[equation]{Proposition}

\newtheorem*{namedtheorem}{\theoremname}
\newcommand{\theoremname}{testing}
\newenvironment{named}[1]{\renewcommand{\theoremname}{#1}\begin{namedtheorem}}{\end{namedtheorem}}

\theoremstyle{definition}
\newtheorem{definition}[equation]{Definition}

\newtheorem{construction}[equation]{Construction}

\newcommand{\RR}{{\mathbb{R}}}

\newcommand{\QQ}{{\mathbb{Q}}}

\newcommand{\refthm}[1]{Theorem~\ref{Thm:#1}}
\newcommand{\reflem}[1]{Lemma~\ref{Lem:#1}}
\newcommand{\refprop}[1]{Proposition~\ref{Prop:#1}}

\newcommand{\refsec}[1]{Section~\ref{Sec:#1}}
\newcommand{\reffig}[1]{Figure~\ref{Fig:#1}}

\newcommand{\bdy}{\partial}

\newcommand{\abs}[1]{\left\vert #1 \right\vert}

\title[Canonical triangulations of Dehn fillings of the Borromean rings]{Canonical triangulations of Dehn fillings of the Borromean rings link complement}

\author{Sophie L. Ham}
\address{School of Mathematics, 
The University of Sydney, NSW 2006, Australia}
\email{sophie.ham@sydney.edu.au}

\begin{document}

\begin{abstract}

The set of canonical decompositions of a cusped hyperbolic $3$-manifold is a complete topological invariant. However, there are only a handful of infinite families for which canonical decompositions are known. In this paper, we find canonical triangulations for Dehn fillings of the Borromean rings link complement and two related manifolds obtained by half-twists. The underlying triangulations were shown to be geometric by Ham and Purcell. To obtain canonical triangulations, we show a local convexity result at every face of the geometric triangulation.

\end{abstract}

\maketitle

%%%%%%%%%%%%%%%%%%%%%%%%%%%%%%%%%%%%%%%%%%%%%%%%%%%%%%%%%%%%%%%%%
\section{Introduction}\label{Sec:Intro}

%This paper is primarily concerned with triangulations of $3$-manifolds: topological, geometric, and canonical. 

A \emph{topological triangulation} of a $3$-manifold $M$ is a subdivision of $M$ into ideal tetrahedra which glue together to give a manifold homeomorphic to $M$. A topological ideal triangulation is all encompassing: the interior of every compact $3$-manifold with torus boundary admits a topological ideal triangulation~\cite{Bing, Moise}.

A \emph{geometric triangulation} is a much stronger notion. It is an ideal triangulation of a cusped hyperbolic $3$-manifold $M$ into positively oriented, finite volume tetrahedra, each of which comes with a hyperbolic structure; gluing the tetrahedra together induces a complete hyperbolic structure on $M$. It remains an open conjecture that every cusped hyperbolic $3$-manifold admits a geometric triangulation. However, the general consensus is in the affirmative, as evidenced by the manifolds of SnapPy~\cite{SnapPy}. Currently, there are only a few families with known geometric triangulations. Gu{\'e}ritaud and Futer found geometric triangulations of $2$-bridge knot complements and punctured torus bundles ~\cite{GueritaudFuter}. Goerner constructed geometric triangulations of manifolds built from isometric Platonic solids~\cite{GoernerII, GoernerI}. On the contrary, Choi found a hyperbolic cone manifold that does not admit a geometric triangulation~\cite{Choi}.

A \emph{canonical triangulation} is even stronger. To define it, let $M$ be a complete hyperbolic $3$-manifold of finite volume, and choose horoball neighbourhoods $H_1, \dots, H_k$ of the cusps $c_1, \dots, c_k$ of $M$. The Ford--Voronoi domain $\mathcal{F}$ is the set of all points in $M$ that have unique shortest path to the union of the horoball neighbourhoods $H_i$. Taking the complement of $\mathcal{F}$ yields a geodesic $2$-dimensional complex $C$. The \emph{canonical decomposition} $\mathcal{D}$ of $M$ relative to $\{H_i\}$ is the \emph{geometric dual} to $C$; that is, the $3$-cells of $\mathcal{D}$ correspond bijectively to the vertices of $C$, the faces of $\mathcal{D}$ correspond bijectively to edges of $C$, and the edges of $\mathcal{D}$ correspond bijectively to the faces of $C$. We will be primarily concerned with determining which manifolds $M$ admit a \emph{canonical triangulation}. This is a manifold for which the canonical decomposition consists only of ideal tetrahedra. 

Epstein and Penner showed that every cusped hyperbolic $3$-manifold of finite volume admits a canonical decomposition into ideal polyhedra \cite{EpsteinPenner}. By extending Epstein and Penner’s construction, Fujii proved that every hyperbolic $3$-manifold of finite volume with non-empty, totally geodesic boundary has a canonical decomposition into partially truncated polyhedra \cite{Fujii}. Akiyoshi showed that the number of decompositions $\mathcal{D}$ produced by Epstein and Penner's method is finite \cite{Akiyoshi}. If $M$ has exactly one cusp, the decomposition is unique.

Let us now mention a motivating reason for studying canonical decompositions. If $M$ is a complete hyperbolic $3$-manifold of finite volume, then the set of canonical decompositions of $M$ is a complete topological invariant by Mostow--Prasad rigidity. Therefore canonical decompositions encode important information about the geometry of a manifold. Unfortunately, it appears to be difficult to compute canonical decompositions in general. SnapPy will compute canonical decompositions of many small examples of $3$-manifolds~\cite{SnapPy}. However, there are only a few classes of infinite families for which canonical decompositions are known. Using symmetry, Sakuma and Weeks constructed canonical decompositions for certain classes of hyperbolic link complements~\cite{SakumaWeeks}. Others found canonical decompositions of once-punctured torus bundles and two-bridge link complements, including Akiyoshi~\cite{Akiyoshi:Ford} using work of Akiyoshi, Sakuma, Wada, and Yamashita~\cite{ASWY}, Lackenby~\cite{Lackenby}, and Gueritaud and Futer~\cite{GueritaudFuter}.

Gu\'{e}ritaud and Schleimer studied canonical decompositions of Dehn fillings~\cite{GueritaudSchleimer}. More specifically, they construct the canonical decomposition of a Dehn filled manifold under certain genericity conditions by gluing a triangulated solid torus into a triangulated manifold. They also find the canonical decompositions of an infinite family of Dehn fillings of one component of the Whitehead link. To obtain canonical decompositions, we will follow the lead of Gu{\'e}ritaud and Schleimer. Their key insight was to check a local convexity condition in Minkowski space at every face of the triangulation.

Our main result on canonical triangulations is the following.

\begin{named}{\refthm{CanonBorromean}}
Let $L$ be a fully augmented link with exactly two crossing circles, as in \reffig{BorromeanRings}. Let $M$ be the manifold obtained by Dehn filling the crossing circles of $S^3-L$ along slopes $m_1, m_2 \in (\QQ\cup\{1/0\})- \{0, 1/0, \pm 1, \pm 2\}$.
Then $M$ admits a canonical triangulation.
\end{named}

The proof of \refthm{CanonBorromean} relies on the underlying triangulation to be geometric. In \cite{JessicaSophie}, we argue that Dehn filling can be performed in such a way that we obtain a geometric triangulation. The method involves constructing different triangulations of solid tori with the correct boundary triangulation. These were first described in \cite{GueritaudSchleimer}, and are shown to be geometric via the theory of angle structures. 

In particular, we prove: 

\begin{named}{\refthm{MainBorromean}}[Ham--Purcell \cite{JessicaSophie}]
Let $L$ be a fully augmented link with exactly two crossing circles, as in \reffig{BorromeanRings}. Let $M$ be the manifold obtained by Dehn filling the crossing circles of $S^3-L$ along slopes $m_1, m_2 \in (\QQ\cup\{1/0\})- \{0, 1/0, \pm 1, \pm 2\}$.
Then $M$ admits a geometric triangulation.
\end{named}

\subsection{Organisation}
This paper is organised as follows.
In \refsec{Back}, we cover the relevant background for the proof of \refthm{CanonBorromean}. In particular, we outline some of the main elements in the proof of \refthm{MainBorromean} as it is necessary to have geometric triangulations. This will involve constructing different triangulations of solid tori. In \refsec{Mink}, we review Minkowski space and the inequality required to prove canonicity. In \refsec{Canon}, we complete the proof of \refthm{CanonBorromean}, on canonical triangulations. In \refsec{GeometryBR}, we describe the required geometry of the unfilled cusp of the Borromean rings link complement. In \refsec{ProvingCanon}, we show canonicity by checking a set of inequalities in Minkowski space.

\subsection{Acknowledgments}
This research was supported in part by an Australian Government Research Training Program (RTP) Scholarship, and in part by an AustMS lift-off fellowship. The author would like to thank Jessica Purcell for helpful conversations.

\section{Background}\label{Sec:Back}

In this section, we review relevant background to complete the proof of \refthm{CanonBorromean}. The results presented here come from previous work by Ham and Purcell \cite{JessicaSophie}.

\subsection{Borromean Rings}

We will consider the Borromean rings link complement as a fully augmented link; see \reffig{BorromeanRings} for a picture. Using the cut and open method for decomposing fully augmented links into polyhedra as in Section~2 of \cite{JessicaSophie}, we show that the Borromean rings link complement decomposes into two regular ideal octahedra. See \reffig{BorrCirclePacking}. 

\begin{figure}
  \centering
  \includegraphics{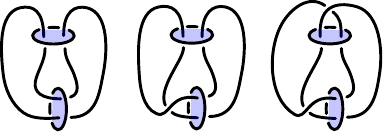}
  \caption{The Borromean rings viewed as a fully augmented link and the other two fully augmented links with exactly two crossing circles.}
  \label{Fig:BorromeanRings}
\end{figure}

In Section~6 of \cite{JessicaSophie}, we prove the following lemma:

\begin{lemma}\label{Lem:BorrDecomp}
Let $L$ be one of the three fully augmented links with exactly two crossing circles, as shown in \reffig{BorromeanRings}. Then $M=S^3-L$ admits a decomposition into two regular ideal octahedra such that one vertex per octahedron meets each crossing circle cusp. If we take one of these vertices to infinity, we get a square that can be arranged in $\mathbb{R}^2$. Put the vertices of one square at $(0,0)$, $(1,0)$, $(1,1)$, and $(0,1)$ in $\RR^2$; put the vertices of other at $(1,0)$, $(2,0)$, $(2,1)$, and $(1,1)$ in $\RR^2$. 
\end{lemma}

If the crossing circle does not encircle a half-twist, then the arc from $(0,0)$ to $(0,1)$ projects to a meridian of the crossing circle. If the crossing circle encircles a half-twist, the arc from $(0,0)$ to $(1,1)$ projects to a meridian. The arc from $(0,0)$ to $(2,0)$ always projects to a longitude of the crossing circle.

\begin{figure}[h]
  \centering
  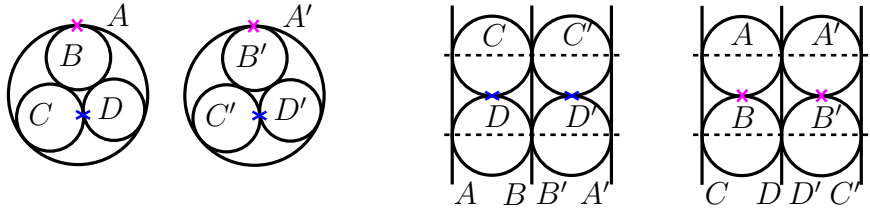
  \caption{Left: Shows the circle packings corresponding to the two regular ideal octahedra. Right: Shows the squares seen in the crossing circle cusps. Figure from \cite{JessicaSophie}.}
  \label{Fig:BorrCirclePacking}
\end{figure}

The proof is completed by considering the decomposition. In particular, by taking pairs of crossing circles to infinity to get the required squares. Note also that the squares in \reflem{BorrDecomp} correspond to the square bases of the two regular ideal octahedra which make up the decomposition. We see one vertex per octahedron in each crossing circle cusp.   

The next lemma can be found in Section~6 of \cite{JessicaSophie}.

\begin{lemma}\label{Lem:BorromeanTet}
Let $L$ be one of the three fully augmented links with exactly two crossing circles, as shown in \reffig{BorromeanRings}. Then $M=S^3-L$ admits a decomposition into two regular ideal octahedra. The square bases as seen in the crossing circle cusps are glued as follows. 

The square on the left of the first cusp glues to the square on the left of the other cusp by a reflection in the diagonal of negative slope. The square on the right of the first cusp glues to the square on the right of the other cusp by a reflection in the diagonal of positive slope; see \reffig{BorrGluing}.
\end{lemma}

\begin{figure}
  \centering
  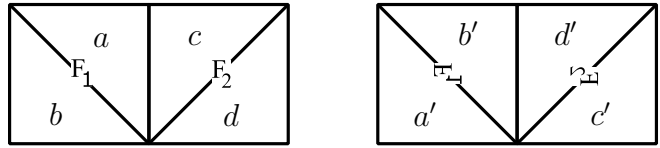
  \caption{How the square bases of the crossing circle cusps are glued to obtain the Borromean rings link complement. Figure from \cite{JessicaSophie}.}
  \label{Fig:BorrGluing}
\end{figure}

\subsection{Layered solid tori and double layered solid tori}\label{Sec:LayeredST}

To perform Dehn filling on the Borromean rings link complement, we require triangulations of solid tori. Jaco and Rubinstein studied triangulations they called \emph{layered solid tori} in \cite{JacoRubinstein:LST}. However, the boundary of layered solid tori have two triangles whose union is a once-punctured torus. We require solid tori with boundary consisting of a twice-punctured torus.

First we describe the procedure for building a layered solid torus. Then we introduce a related construction for building a \emph{double layered solid torus} — one with boundary consisting of a twice-punctured torus. Gu{\'e}ritaud and Schleimer first considered double layered solid tori in their work on the Whitehead link \cite{GueritaudSchleimer}. 

First we review an ideal triangulation of $\mathbb{H}^2$ called the \emph{Farey triangulation}. Using the Poincaré disk model, we place $1/1$ at the North Pole, and $-1/1$ at the South Pole. Put antipodal points $0/1$ and $1/0=\infty$ to the east and west, respectively; see \reffig{LLFarey}. Two rational slopes $a/b$ and $c/d$ in $\QQ \cup \{1/0\} \subset \bdy \mathbb{H}^2$ are connected by an ideal geodesic whenever $i(a/b, c/d) = |ad - bc|=1$, where $i(a/b, c/d)$ denotes the intersection number of slopes on the torus. 

\begin{figure}[h]
  \centering
  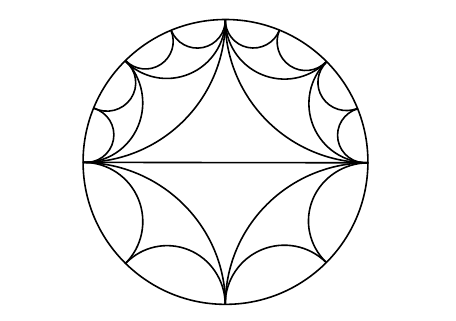
  \caption{The Farey triangulation.}
  \label{Fig:LLFarey}
\end{figure}

The triangulation of a once-punctured torus consists of three boundary slopes such that the intersection number of each pair of slopes is $1$. Therefore each triangle in the Farey graph determines a triangulation of a once-punctured torus. In fact, the Farey triangulation completely parametrises the space of all triangulations of once-punctured tori. Since the boundary of a layered solid torus is a triangulated once-punctured torus, it turns out we may use the combinatorics of the Farey graph to build such a layering. 
We will take a walk in the Farey graph to construct the triangulation. The walks begins in the centre of the \emph{initial triangle} denoted by $T_0$, and ends at our chosen slope $m \subset \partial \mathbb{H}^2$. Note that $T_0$ could be either $\{0/1, 1/0, 1/1\}$ or $\{0/1, 1/0, -1/1\}$. Along the way, we cross a sequence of triangles in the Farey triangulation, labelled as follows:
\[ (T_0, T_1, \dots, T_N) = (pqr, pqr’, \dots, stm). \]

Note that we require the path to pass through at least two triangles, i.e. $N \geq 2$. That is, $m \notin \{0/1, 1/0, \pm 1/1, \pm 2, \pm 1/2\}$. Moreover, the path is not allowed to cross the same triangle twice. Depending on the direction of our path in the Farey graph at each step, we denote a left turn by $L$ and a right turn by $R$. The first two steps are assigned either $LL$ or $RR$.

The first layer of the solid torus is built by attaching an ideal tetrahedron to the once-punctured torus with triangulation determined by $T_0$. As we pass from $T_0$ to $T_1$, notice that one of the three boundary slopes of $T_0$ is replaced with one of the slopes of $T_1$. Thus two consecutive once-punctured tori have two edge slopes that are the same and two that differ. In particular, the next once punctured torus is obtained from the previous one by identifying edges of the same slope and exchanging the two edges of different slopes. The exchange that takes place is called a \emph{diagonal exchange}. A tetrahedron is inserted each time we perform a diagonal exchange. That is, each time we pass from one Farey triangle to the next in our path, we layer on a tetrahedron. 

Inductively, at the $k$-th step, we layer on a tetrahedron via a diagonal exchange. The current layering of tetrahedra is now homotopy equivalent to a thickened solid torus $T^2 \times I$ with boundary consisting of two once-punctured tori, whose triangulations are each given by $T_0$ and $T_k$. We continue adding tetrahedra until we reach $k=N-1$. Note that we do not layer on a tetrahedron as we pass from $T_{n-1}$ to $T_n$. Instead we will close up the space at this step. 

To obtain a solid torus with homotopically trivial slope $m$, we close up one of the boundary components of the thickened torus. When we reach $k=N-1$, one of the boundary components is a once-punctured torus triangulated by two ideal triangles with slopes coming from $T_{n-1}$. We identify the two ideal triangles by folding across the edge of slope $m’$, so that the edge of slope $m$ bounds a disc; see \reffig{FoldM}. The result is a layering of tetrahedra homeomorphic to a solid torus with boundary triangulation corresponding to $T_0$. 

\begin{figure}[h]
  \centering
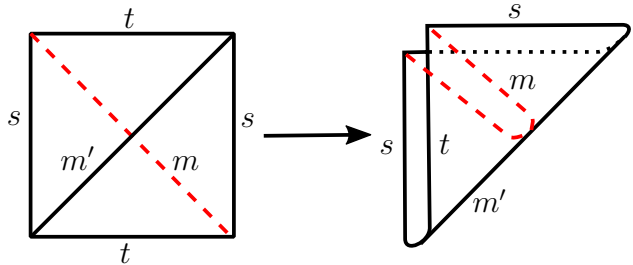
  \caption{Folding procedure across the edge $m’$ so that $m$ bounds a disc.}
  \label{Fig:FoldM}
\end{figure}

\begin{construction}\label{Constr:SideBySide}

Let $m=l/k$ be a slope such that $l=2s$ is even for some integer $s$, and $m \notin \{0/1, \pm 2/1\}$. Let $(T_0, \dots, T_N)$ be the sequence of triangles we encounter on a path in the Farey graph starting in the initial triangle $T_0$ with slopes $0/1, 1/0, \pm 1/1$, and ending in the triangle $T_N$ with slopes $u, t$ and $s/k$.

The \emph{side-by-side construction} presented here is analogous to the one for a layered solid torus, however, now we layer on two identical ideal tetrahedra at a time, arranged side-by-side. This will give the correct boundary triangulation of a twice-punctured torus. The tetrahedra at each step are formed in exactly the same way as a solid torus by performing a diagonal exchange. 

To begin, we layer on two identical tetrahedra by gluing their faces to a twice-punctured torus whose triangulation corresponds to two copies of $T_0$, arranged side-by-side. See \reffig{SideBySide}. Inductively, after the $k$-th step, we obtain a layering of tetrahedra homotopy equivalent to an interval times a twice-punctured torus. The resulting space has two boundary components consisting of twice punctured torus; their triangulations coming from $T_0$ and $T_k$. This time we continue to perform diagonal exchanges until $k=N$, introducing a tetrahedron at each step. This differs from a layered solid torus where we stop at the point $k=N-1$.

\begin{figure}
  \centering
  \includegraphics{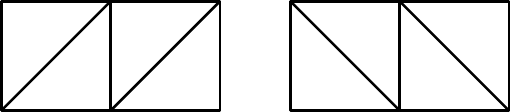}
  \caption{The boundary triangulations of a double layered solid torus.}
  \label{Fig:SideBySide}
\end{figure}

It remains to close up the current layering of tetrahedra. To homotopically kill the slope $m$, we identify the two edges of slope $s/k$ in the boundary component triangulated by two copies $T_n$. This introduces an extra tetrahedron at the core whose four faces are glued to the faces of the innermost layer. Note that the additional tetrahedron in this construction means we only need to exclude slopes in $T_0$. That is, $s/k \notin \{0/1, \pm 1/1\}$.

Suppose we take the cover $\mathbb{R}^2$ of the boundary of a double layered solid torus, so that punctures correspond to points $\mathbb{Z}^2 \subset \mathbb{R}^2$. A meridian $\mu$ of slope $1/0$ lifts to the curve running from $(0,0)$ to $(0,1)$. A longitude $\lambda$ of slope $0/1$ lifts to run from $(0,0)$ to $(2,0)$. Then a slope on the boundary of the torus $m=l/k$ lifts to the curve running from $(0,0)$ to $(2k, l)$.

If $l$ is odd, the lift of $m=l/k$ avoids hitting any of the lifts of punctures $\mathbb{Z}^2$ in its interior. In this case, we may take the double cover of a layered solid torus. If $l=2s$ is even for some integer $s$, then the slope $m=l/k$ hits a lift of a puncture in its interior at $(k,s) \in \mathbb{Z}$; taking the double cover of a layered solid torus will not work. In this case, we use the side-by-side construction to build the triangulation. Thus the slope $l/k$, depending on whether $l$ is odd or even, determines the triangulation of a double layered solid torus. See Lemma 7.1 and Lemma 7.4  in \cite{JessicaSophie} for more details on the above discussion. 

\end{construction}

\subsection{Dehn fillings of the Borromean Rings}\label{Sec:DFBorromean}

In Section 8 of \cite{JessicaSophie}, we prove: 

\begin{theorem}\label{Thm:MainBorromean}
Let $L$ be a fully augmented link with exactly two crossing circles, as in \reffig{BorromeanRings}. Let $M$ be the manifold obtained by Dehn filling the crossing circles of $S^3-L$ along slopes $m_1, m_2 \in (\QQ\cup\{1/0\})- \{0, 1/0, \pm 1, \pm 2\}$.
Then $M$ admits a geometric triangulation.
\end{theorem}

%\begin{lemma}\label{Lem:TriangBorFilling}
%Let $M$ be one of the fully augmented link complements with exactly two crossing circles. 
%Let $m_1, m_2$ be slopes such that
%\[ m_1, m_2 \notin \{ 0/1, 1/0, \pm 1/1, \pm 2/1 \}. \]
%Then the Dehn filling of $M$ on its crossing circle cusps along slopes $m_1$ and $m_2$, denoted $M(m_1,m_2)$, admits a topological triangulation built by gluing together two triangulated solid tori that are either both double covers of layered solid tori, or one double cover and one solid torus with the side-by-side construction of \reflem{DoubleLST}, or two solid tori of that form.
%\end{lemma}

The triangulation is obtained as follows. By \reflem{BorrDecomp}, the Borromean rings link complement decomposes into two ideal octahedra. Moreover, each crossing circle cusp is tiled by the two squares corresponding to the square bases of the two octahedra. The square bases glue together as in \reflem{BorromeanTet}. To perform Dehn filling, we remove a portion of the octahedra meeting each crossing circle cusp from the manifold, leaving only the two square bases. We attach triangulated solid tori to these square bases. As discussed in \refsec{LayeredST}, the slopes $m_1$ and $m_2$ on the crossing circle cusps each determine a triangulation of a solid torus. That is, we attach either both double covers of layered solid tori, or one double cover and one solid torus with the side-by-side construction discussed in \refsec{LayeredST}, or two solid tori of that form. The triangulation is shown to be geometric in \cite{JessicaSophie} by applying the theory of angle structures; the reader is referred to Section 7 and Section 8 of \cite{JessicaSophie} for more details. 

\section{Minkowski space}\label{Sec:Mink}
In this section, we review Minkowski space. The material covered here will be relevant in \refsec{Canon}, where we complete the proof of \refthm{CanonBorromean}, on canonical triangulations. For a description of the five models of hyperbolic space, including the Minkowski space model, see the article by Cannon, Floyd, Kenyon and Parry \cite{CannonFloydKenyonParry}. The reader is also referred to Section~4.3 of Gu\'{e}ritaud-Schleimer \cite{GueritaudSchleimer}.

\emph{Minkowski space} $\mathbb{M}^{n+1}$ is the real vector space $\mathbb{R}^{n+1}$ equipped with the inner product \[ \langle x, y \rangle = x_1y_1+ \dots + x_ny_n - x_{n+1}y_{n+1}\] for $x = (x_1, \dots, x_n, x_{n+1}) \in \mathbb{R}^{n+1}$ and $y = (y_1, \dots, y_n, y_{n+1}) \in \mathbb{R}^{n+1}$. \emph{The Minkowski space model} of hyperbolic $n$-space is defined to be the set \[ X = \{x = (x_1, \dots, x_n, x_{n+1}) \in \mathbb{M}^{n+1} \mid \langle x, x \rangle = -1 , x_{n+1} > 0\}. \] The product $\langle \, \cdot \,, \cdot \, \rangle$ restricts to the Riemannian metric on $X$.

Define the isotropic cone to be the set \[  L = \{v \in \mathbb{M}^{n+1} \mid \langle v, v \rangle = 0\} \subset \mathbb{M}^{n+1}. \] 
A \emph{horoball} in $X$ is the set \[ H_v = \{ x \in X \mid \langle v, x \rangle \geq -1\} \] for $v \in L$. Since $H_v = H_v'$ implies $v = v'$, we may associate a unique isotropic vector $v(H) \in L$ to each horoball $H_{v(H)}$ in $X$.

To move between the Minkowski space model and the upper half-space model, we use the following isometry. Consider the Minkowski space model in $\mathbb{R}^4$. Then \[ X = \{ x = (x_1, x_2, x_3, x_4) \mid \langle x, x \rangle = -1, x_4 > 0 \}. \] Identify the point $(x_1, x_2, x_3, x_4)$ in Minkowski space with the point in the upper half space model of $\mathbb{H}^3$ lying over the complex number $(x_1 + i x_2)/(x_3 + x_4)$ at Euclidean height $1/(x_3 + x_4)$.

Let $H_{d, \zeta}$ be the horoball of Euclidean diameter $d$ about the point $\zeta = \xi + i\eta \in \mathbb{C}$ in the upper half-space model of $\mathbb{H}^3$. Under the above isometry, we identify $H_{d, \zeta}$ with the unique isotropic vector 
\begin{equation}\label{Eq:IsoVec}
v_{d, \zeta} = \dfrac{1}{d} (2\xi, 2\eta, 1-\abs{\zeta}^2, 1+\abs{\zeta}^2) 
\end{equation} 
corresponding to the horoball $\{x \in X \mid \langle v_{d, \zeta}, x \rangle \geq -1\}$ in the Minkowski space model of $\mathbb{H}^3$. So our usual horoballs in the upper half-space model correspond to isotropic vectors in the Minkowski space model. The horoball $H_{h, \infty}$ of Euclidean height $h$ about $\infty$ in the upper half-space model is identified with the isotropic vector $v_{h, \infty} = (0, 0, -h, h)$ corresponding to the horoball $\{x \in X \mid \langle v_{h, \infty}, x \rangle \geq -1\}$ in Minkowski space. Similarly, an infinite set of horoballs $(H_i)_{i \in I}$ in $\mathbb{H}^3$ corresponds to an infinite set of isotropic vectors $(v_i)_{i \in I}$ in Minkowski space. 

Let $M$ be a complete hyperbolic $3$-manifold of finite volume, and choose horoball neighbourhoods $H_1, \dots, H_k$ of the cusps $c_1, \dots, c_k$ of $M$. Suppose we lift the horoballs $H_i$ to an infinite set of horoballs $(H_i)_{i \in I}$ in $\mathbb{H}^3$, then under the above identification, we obtain a corresponding infinite set of isotropic vectors $V = (v_i)_{i \in I}$ in Minkowski space. In \cite{EpsteinPenner}, Epstein--Penner use a convex hull construction in Minkowski space to obtain a canonical decomposition of $M$. In particular, they show that the boundary of the convex hull of $V$ admits a decomposition $\tilde{\mathcal{D}}$ into convex polyhedra. Projecting the faces of $\tilde{\mathcal{D}}$ to $\mathbb{H}^3$ and then to $M$ gives a canonical decomposition $\mathcal{D}$ of $M$.

Conversely, we can determine whether a given ideal decomposition $\mathcal{D}$ of $M$ is canonical in Minkowski space. The decomposition $\mathcal{D}$ lifts to a decomposition $\hat{\mathcal{D}}$ in $\mathbb{H}^3$ and this corresponds to a decomposition $\tilde{\mathcal{D}}$ in Minkowski space. Provided $\tilde{\mathcal{D}}$ is locally convex at every $2$-dimensional face, the decomposition $\mathcal{D}$ is \emph{geometrically canonical}.

As the next proposition demonstrates, proving local convexity at a face in $\tilde{\mathcal{D}}$ is simply a matter of checking an inequality in Minkowski space. This result is Proposition~25 of Gu{\'e}ritaud--Schleimer \cite{GueritaudSchleimer}.
 
\begin{proposition}\label{Prop:LocalConvexity}
Suppose $M$ admits an ideal decomposition $\mathcal{D}$ into polyhedra, corresponding to the polyhedral complex $\tilde{\mathcal{D}}$ in Minkowski space. Let $F = A_1 \cdots A_{\sigma}$ be a $2$-dimensional face of $\tilde{\mathcal{D}}$ with vertices, $A_1, \dots, A_{\sigma}$, with $P \notin F$ and $Q \notin F$ vertices of the faces adjacent to $F$. Then the polyhedral complex $\tilde{\mathcal{D}}$ is locally convex at $F$ if and only if the following inequality holds:
\[ \rho P + (1-\rho) Q = \sum_{i=1}^{\sigma} \lambda_i A_i \text{ where } \rho \in (0,1) \text{ and }  \sum_{i=1}^{\sigma} \lambda_i > 1. \]
\end{proposition}

\section{Canonical decompositions}\label{Sec:Canon}

In this section, we prove \refthm{CanonBorromean}, on canonical triangulations of Dehn fillings of the Borromean rings link complement. We will follow the lead of Gu{\'e}ritaud and Schleimer who found canonical triangulations of Dehn fillings on one cusp of the Whitehead link complement~\cite{GueritaudSchleimer}. In particular, we will check local convexity at every face in the triangulation of \refthm{MainBorromean} via an inequality in Minkowski space; see \refprop{LocalConvexity} for details. In \refsec{GeometryBR}, we describe the geometry of the remaining cusp after Dehn filling two components of the Borromean rings link complement. \refsec{ProvingCanon} is dedicated to proving canonicity by checking the required set of convexity statements in Minkowski space.

\subsection{Geometry of the unfilled cusp}\label{Sec:GeometryBR}

Let $M$ be the manifold obtained by Dehn filling the crossing circle cusps of the Borromean rings link complement along slopes $m_1, m_2 \in (\QQ\cup\{1/0\})- \{0, 1/0, \pm 1, \pm 2\}$. In this section, we describe the geometry of the unfilled cusp of $M$. Recall that the triangulation of $M$ given by \refthm{MainBorromean} is obtained by gluing together two triangulated solid tori that are either both double covers of layered solid tori, or one double cover and one solid torus with the side-by-side construction, or two solid tori of that form; see \refsec{DFBorromean}.

\begin{lemma}
The cusp torus remaining after Dehn filling is tiled by four hexagons: two coming from one solid torus, two from the other. There is always a non-convex hexagon adjacent to a convex one, alternating throughout the cusp torus. See \reffig{SnapPyFourHexagons} and \reffig{FourHexagons}.
\end{lemma}

\begin{figure}
  \centering
  \includegraphics[width=\textwidth]{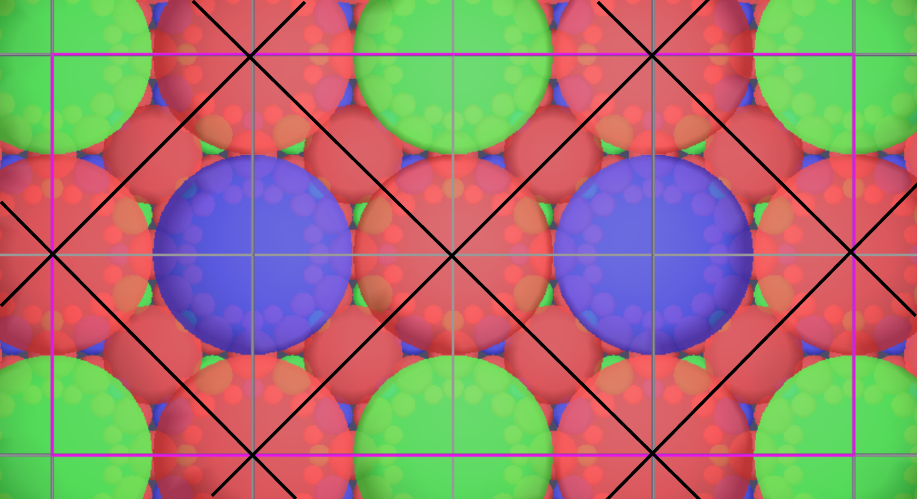}
  
  \vspace{1cm}
  
  \includegraphics[width=\textwidth]{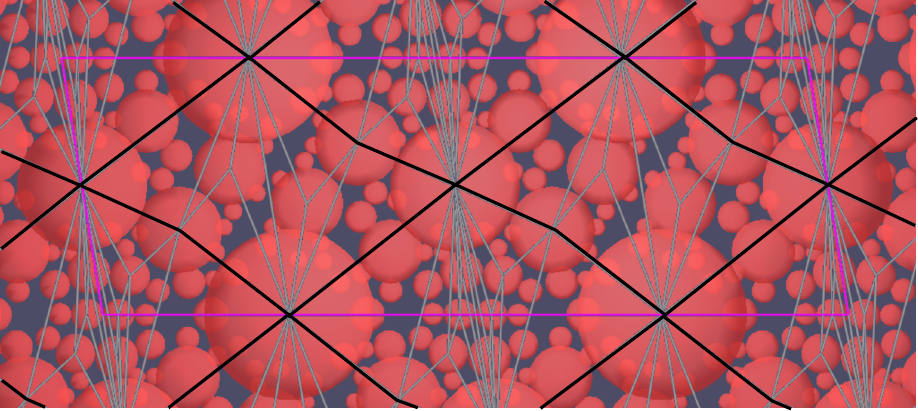}
  \caption{Top is a screenshot from SnapPy showing the unfilled cusp before the Dehn filling. Bottom shows the cusp after Dehn filling along slopes $(1,3)$ and $(1,2)$.}
  \label{Fig:SnapPyFourHexagons}
\end{figure}
  
\begin{figure}
  \centering
  \includegraphics{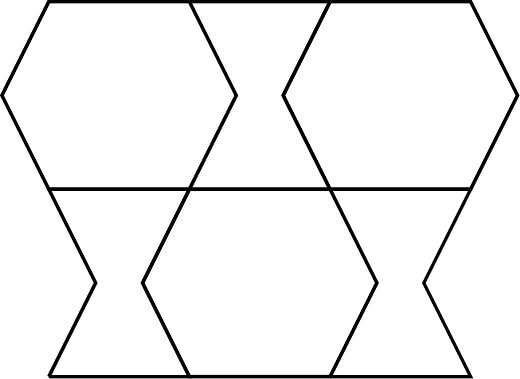}
  \caption{Schematic drawing after Dehn filling of the hexagons tiling the unfilled cusp.}
  \label{Fig:FourHexagons}
\end{figure}

\begin{proof}

Recall that the Borromean rings link complement decomposes into two ideal octahedra such that one vertex per octahedron meets each crossing circle cusp; see \reflem{BorrDecomp}. As such, each crossing circle cusp is tiled by two squares. These correspond to the square bases of the ideal octahedra; the squares are shown in \reffig{BorrGluing}, and identified as in \reflem{BorromeanTet}. The Dehn filling is performed by removing a portion of the octahedra, in fact removing all but the two square bases from the manifold and attaching triangulated solid tori to these square bases. The squares are then triangulated by choosing either both positive or both negative diagonals, as in \reffig{SideBySide}. 

We can view the octahedra in the unfilled cusp as follows. Each of the square bases meets four ideal vertices all identified to the cusp that won’t be filled. These correspond to three red horoballs in \reffig{SnapPyFourHexagons} (top) as well as the one at infinity. In particular, each square forms a line running through the centres of three red horoballs, lying between the blue and the green horoballs. These lines cut out diamond shapes in \reffig{SnapPyFourHexagons} (top), giving four distinct diamonds, each of the same shape. 
 
When we select the diagonal slopes $\pm 1$ of the square bases, these appear as geodesics from infinity to the centres of the smaller red horoballs between the blue and the green horoballs. Thus the diagonal slopes lie on totally geodesic vertical planes meeting the unfilled cusp in diamond shapes. So the four hexagons look like diamonds before Dehn filling.

After Dehn filling, the diagonal slopes $\pm 1$ bend either inwards or outwards depending on the signs $m_1, m_2$ of our slopes of the Dehn filling. This always results in a non-convex hexagon adjacent to convex one; see \reffig{SnapPyFourHexagons} (bottom) and \reffig{FourHexagons}. The three cases are shown in \reffig{Borr3Cases}.

Finally recall that the boundary of one solid torus is glued to the boundary of the other by a reflection in the diagonal; see \reflem{BorromeanTet}. The six sides of each hexagon inherit this gluing, resulting in an alternating pattern throughout the unfilled cusp; see again \reffig{SnapPyFourHexagons} (bottom).

\end{proof}

\begin{lemma}\label{Lem:OneFace}
To prove geometric canonicity at every face on the boundary of the Borromean rings link complement, we only need to show the convexity condition of \refprop{LocalConvexity} holds for one pair of adjacent cusp triangles, across two adjacent hexagons in the unfilled cusp.
\end{lemma}

\begin{proof}
Each boundary face of a double layered solid torus appears three times on the boundary of a hexagon as a lift of that face. So we only need to show the convexity condition of \refprop{LocalConvexity} holds for one lift of that face. This proves local convexity at three faces, alternating every other one around the hexagon. But the hexagon is symmetric, so the same argument for one face applies immediately to the opposite face, which is the second side of the hexagon. 
\end{proof}

\begin{lemma}\label{Lem:ThreeCases}
There are only three ways to match up one pair of adjacent cusp triangles, across two adjacent hexagons in the cusp that won't be filled.
\end{lemma}

\begin{figure}
  \centering
  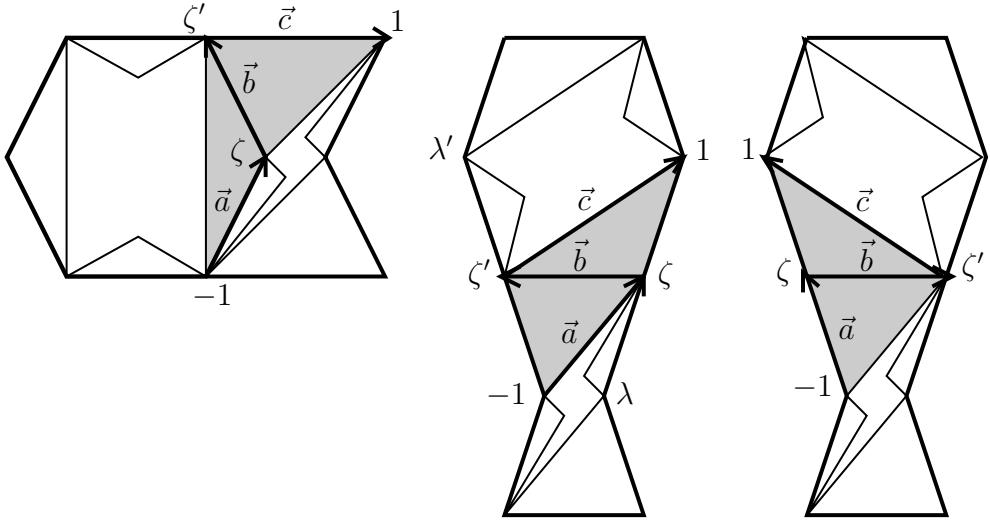
  \caption{Shown are the three cases of pairs of triangles that must be checked for \refprop{LocalConvexity}. Left: The case that the slopes $m_1$ and $m_2$ have opposite signs. Middle: The case that the slopes $m_1$ and $m_2$ are both negative. Right: The case that the slopes $m_1$ and $m_2$ are both positive.}
  \label{Fig:Borr3Cases}
\end{figure}

\begin{proof}
Recall that when we build a layered solid torus, we have the choice to cover one of three slopes on the boundary. This corresponds to three distinct triangulations of the non-convex hexagon in the unfilled cusp. For the convex hexagon, there are only two possible triangulations and these are symmetric. There are only three ways to match up the triangulations of the non-convex and convex hexagons; these are shown in Figure~\ref{Fig:Borr3Cases}.
\end{proof}

\subsection{Proving geometric canonicity}\label{Sec:ProvingCanon}

Let $M$ be the manifold obtained by Dehn filling the crossing circle cusps of the Borromean rings link complement along slopes $m_1, m_2 \in (\QQ\cup\{1/0\})- \{0, 1/0, \pm 1, \pm 2\}$. The geometric triangulation of $M$ given by \refthm{MainBorromean} is obtained by gluing together two triangulated solid tori. In this section, we prove that this triangulation is canonical. We need to prove geometric canonicity at every face in the triangulation, so there are a number of cases consider. In particular, we need to show the convexity condition of \refprop{LocalConvexity} holds when the face is:
\begin{itemize}
\item[(1)] a face on the boundary of the layered solid torus; this is the only face also in the Borromean rings link complement. Thus we refer to it below as a face on the boundary of the Borromean rings link complement.
\item[(2)] an interior face of a double cover of a layered solid torus,
\item[(3)] a face at the core of a double cover of a layered solid torus,
\item[(4)] an interior face of a side-by-side layered solid torus, 
\item[(5)] a face at the core of a side-by-side layered solid torus. 
\end{itemize}

Gu\'{e}ritaud and Schleimer prove geometric canonicity at an interior face and at the core of a layered solid torus \cite[Section~4.4]{GueritaudSchleimer}. They note that this extends to doubled layered solid tori, showing (2) through (5). Since they only sketch the argument, we work through it carefully here. As we shall see later, their argument can also be applied to show geometric canonicity at a face on the boundary of the Borromean rings link complement, item (1).

Recall that the boundary of a layered solid torus forms an outermost hexagon in the cusp diagram; a path around the puncture on the boundary meets exactly six triangles, and these form a hexagon. A new hexagon is introduced in the interior each time a tetrahedron is added. Label the vertices of the hexagon between the tetrahedra $\Delta_i$ and $\Delta_{i+1}$ as $-1, \zeta, \zeta', 1, -\zeta, -\zeta'$; see \reffig{GSHex}. We may assume without loss of generality that the tetrahedron $\Delta_{i+1}$ is added in the $L$ direction.

\begin{figure}
  \centering
  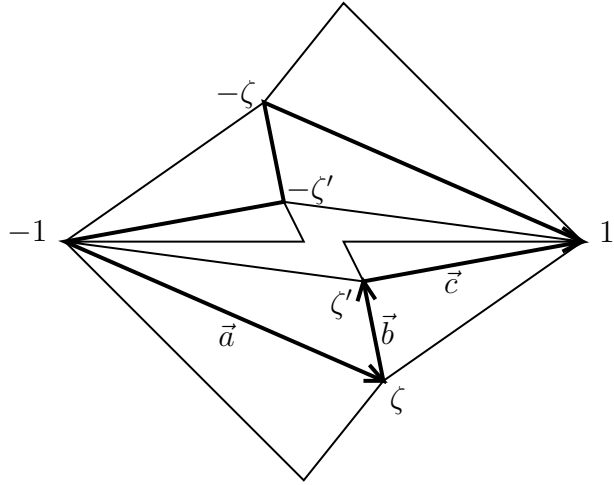
  \caption{The hexagon between the tetrahedra $\Delta_i$ and $\Delta_{i+1}$.}
  \label{Fig:GSHex}
\end{figure}

Following Gu\'{e}ritaud and Schleimer, define the following three vectors
\begin{align*}
  \zeta + 1 & = \vec{a} = a\exp(iA) \\
  \zeta' - \zeta & = \vec{b} = b\exp(iB) \\
  1 - \zeta' & = \vec{c} = c\exp(iC),
\end{align*}
where $a, b, c$ are positive real numbers. The angles $A, B, C$ are defined modulo $2\pi$. 

Each boundary face of a layered solid torus appears three times on the boundary of the outermost hexagon, alternating around the hexagon. This is also true for any hexagon in the interior because we can strip off the outer layers of tetrahedra to obtain a smaller layered solid torus. So there exists a M\"{o}bius transformation $f$ that maps the face $(-1\zeta\infty)$ to the face $(\zeta'1\infty)$. Then $f$ has the form
\[ f: u \mapsto 1 + \dfrac{(\zeta + 1)(\zeta' - 1)}{u + 1} = 1 - \dfrac{\vec{a}\vec{c}}{u + 1}. \]

Throughout this section, we choose an initial horoball $H_{1, \infty}$ of height $1$ centred at $\infty$ and its images under covering transformations. 

\begin{lemma}\label{Lem:diamac}
The horoball about the vertex $1$ of the hexagon has Euclidean diameter $ac$. In other words, the diameter is the product of the lengths of the adjacent vectors in the hexagon.
\end{lemma}

\begin{proof}
Since $f(\infty) = 1$, there exists a horoball about the vertex $1$ in the hexagon. We wish to find its diameter. Consider the isometric sphere between the horoball $H_{1, \infty}$ of height $1$ about infinity and the horoball $f(H_{1, \infty})$. The point directly over the centre of the isometric sphere on the isometric sphere lies at a distance $d$ from $H_{1, \infty}$ and at a distance $d$ from $f(H_{1, \infty})$. Let $p$ be the point at the top of the isometric sphere. Then \[d = \int^1_p \dfrac{1}{y} \,dy = \ln\left(\dfrac{1}{p}\right), \quad \text{or} \quad p = \exp(-d). \] Let $q$ be the diameter of the horoball $f(H_{1, \infty})$. Then \[d = \int^p_q \frac{1}{y} \,dy = \ln(p) - \ln(q) = -d - \ln(q),\] which implies that $q = \exp(-2) = p^2$. Since the isometric sphere has diameter $\sqrt{\abs{\vec{a}\vec{c}}}$, the diameter of the horoball $f(H_{1, \infty})$ is $(\sqrt{\abs{\vec{a}\vec{c}}})^2 =\abs{\vec{a}\vec{c}} = ac$. Hence the horoball $H_{1, \infty}$ maps to $H_{ac, 1}$ under $f$. In other words, the horoball about the vertex $1$ of the hexagon has Euclidean diameter $ac$.
\end{proof}

Let $\alpha$ be the vertical edge with vertices $1$ and $\infty$. By symmetry of the hexagon, there is another edge labelled $\alpha$ with vertices at $-1$ and $\infty$. By symmetry again, the horoball at the end of the edge $\alpha$ with vertices $1$ and $\infty$ has the same diameter as the one at the end of the other edge labelled $\alpha$. That is, given $H_{ac, 1}$, there exists a horoball $H_{ac, -1}$ about the vertex $-1$ with diameter $ac$. 

Similarly, there are two vertical edges $\beta$ with endpoints at $\zeta$ and $-\zeta$ running into $\infty$. 

\begin{lemma}\label{Lem:diamab}
The horoball about the vertex $\zeta$ of the hexagon has Euclidean diameter $ab$. In other words, the diameter is the product of the lengths of the adjacent vectors in the hexagon.
\end{lemma}

\begin{proof}
The argument above was written for the edge $\alpha$, but note it is independent of the choice of $\alpha$. We can rotate and scale by an isometry $r$ such that the vertices $-\zeta$ and $\zeta$ (the endpoints of $\beta$) are sent to $-1$ and $1$. Then $f$ will have the form \[ f: u \mapsto 1 - \dfrac{r(\vec{a})r(\vec{b})}{u + 1} \] where $r(\vec{a})$ and $r(\vec{b})$ are now the vectors adjacent to $\beta$. As above, this sends $H_{1, \infty}$ to $H_{\abs{r(\vec{a}) r(\vec{b})}, 1}$. Now rotate and scale back via $r^{-1}$ so that $1 \mapsto \zeta, r(\vec{a}) \mapsto \vec{a}$ and $r(\vec{b}) \mapsto \vec{b}$. It follows that the horoball about the vertex $\zeta$ of the hexagon has Euclidean diameter $ab$.
\end{proof}

By symmetry of the hexagon, the horoball at the end of the edge $\beta$ with vertices $\zeta$ and $\infty$ has the same diameter as the one at the end of the other edge labelled $\beta$. So there exists a horoball $H_{ab, -\zeta}$ about $-\zeta$ with diameter $ac$. A similar argument as above gives $H_{bc, \zeta'}$ and $H_{bc, -\zeta'}$. 

The horoballs about the vertices $-1, \zeta, \zeta', 1, -\zeta, -\zeta'$ of the hexagon are respectively as follows

\[ H_{ac, -1}, \quad H_{ab, \zeta}, \quad H_{bc, \zeta'}, \quad H_{ac, 1}, \quad H_{ab, -\zeta}, \quad H_{bc, -\zeta'}. \] 

To prove geometric canonicity at the face $(\zeta\zeta'\infty)$ of the hexagon, we set up the convexity inequality of \refprop{LocalConvexity} accordingly. Consider the horoballs centred at the vertices $\zeta, \zeta', \infty$ of this face, and the horoballs adjacent to this face centred at $-1$ and $1$; these are given by

\[ H_{ab, \zeta}, \quad H_{bc, \zeta'}, \quad H_{1, \infty}, \quad H_{ac, -1}, \quad H_{ac, 1}. \] 

Let us now make our move to Minkowski space.

\begin{lemma}
Let $\zeta = \xi + i\eta$ and $\zeta' = \xi' + i\eta'$ with $\xi, \xi', \eta, \eta' \in \mathbb{R}$, then the isotropic vectors corresponding to the horoballs above are respectively 
\[ \begin{array}{ccccccccc}
\vspace{2mm}
v_\zeta & = & \frac{1}{ab} & ( & 2\xi, & 2\eta, & 1 - \abs{\zeta}^2, & 1 + \abs{\zeta}^2 & ) \\ \vspace{2mm}
v_{\zeta'} & = & \frac{1}{bc} & ( & 2\xi', & 2\eta', & 1 - \abs{\zeta'}^2, & 1 + \abs{\zeta'}^2 & ) \\ \vspace{2mm}
v_\infty & = & & ( & 0, & 0, & -1, & 1 & ) \\ \vspace{2mm}
v_{-1} & = & \frac{1}{ac} & ( & -2, &  0, &  0, & 2 & ) \\ \vspace{2mm}
v_{1} & = & \frac{1}{ac} & ( & 2, & 0, & 0, & 2 & ). \\
\end{array} \]
\end{lemma}

\begin{proof}
One can easily check the above correspondence holds using equation~\eqref{Eq:IsoVec} in \refsec{Mink}.
\end{proof}

\begin{lemma}\label{Lem:A+C/2}
Assume that $\lambda v_\zeta + \mu v_{\zeta'} + \nu v_\infty = \rho v_1 + (1 - \rho)v_{-1}$. Then $\lambda + \mu + \nu > 1$ if and only if \[ -\pi < \dfrac{A + C}{2} < 0. \]
\end{lemma}

\begin{proof}
Solving the $4 \times 4$ system of equations given by $\lambda v_\zeta + \mu v_{\zeta'} + \nu v_\infty = \rho v_1 + (1 - \rho)v_{-1}$ yields the unique solution \[ \lambda = \dfrac{b\eta'}{c(\eta' - \eta)}, \quad \mu = \dfrac{-b\eta}{a(\eta' - \eta)}, \quad \nu = \dfrac{\eta'(1 - \abs{\zeta}^2) - \eta(1 - \abs{\zeta'}^2)}{ac(\eta' - \eta)}. \] Note that we omit $\rho$ because it is not required. Then \[ \lambda + \mu + \nu = 1 + \dfrac{Z}{ac(\eta' - \eta)}, \text{ where } Z = ab\eta' - bc\eta + \eta'(1 - \abs{\zeta}^2) - \eta(1 - \abs{\zeta'}^2) - ac(\eta' - \eta). \] Since the triangles $-1\zeta\zeta'$ and $1\zeta'\zeta$ are counterclockwise oriented, we have $\eta' > \eta$. Thus it remains to show that $Z>0$ if and only if \[ -\pi < \dfrac{A + C}{2} < 0. \]

It is easy to verify that $1 - \abs{\zeta}^2 = \vec{a} \cdot (\vec{b}+\vec{c})$ and $1 - \abs{\zeta'}^2 = (\vec{a} + \vec{b})\cdot \vec{c}$, where the operation $\cdot$ denotes the standard dot product between two vectors. Therefore 
\begin{align*}
Z &= \eta'(ab + \vec{a} \cdot \vec{b}) - \eta(bc +\vec{b} \cdot \vec{c}) - (\eta' - \eta)(ac - \vec{a} \cdot \vec{c}) \\
&= abc\left(\dfrac{\eta'}{c}(1 + \cos (A - B)) - \dfrac{\eta}{a}(1 + \cos (B - C)) - \dfrac{\eta' - \eta}{b}(1 - \cos (A - C))\right).
\end{align*}
Observe that $\sin C = -\dfrac{\eta'}{c}, \; \sin A = \dfrac{\eta}{a}$ and $\sin B = \dfrac{\eta' - \eta}{b}$. Hence 
\begin{align*}
Z &= -abc\left(\sin C (1 + \cos (A - B)) + \sin A (1 + \cos (B - C)) + \sin B (1 - \cos (A - C))\right) \\
&= -4abc \; \sin \dfrac{A + C}{2} \cos \dfrac{B - A}{2} \cos \dfrac{B - C}{2}.
\end{align*}

One can check that the last line follows using standard trigonometric formulae. Recall that the angles $A, B$ and $C$ were defined modulo $2\pi$; pick $B$ to be the smallest positive representative. The fact that the triangles $-1\zeta\zeta'$ and $1\zeta'\zeta$ are counterclockwise oriented implies $0 < B < \pi$. 

Since the triangle $1\zeta'\zeta$ is counterclockwise oriented, we must have $0 < B - A < \pi$. This is easy to check: if $B - A$ is equal to $0$ or $\pi$, then the vectors $\vec{a}$ and $\vec{b}$ are parallel; if $B - A < 0$ or $B - A > \pi$, the triangle $1\zeta'\zeta$ is  clockwise oriented. It follows that \[\cos \dfrac{B - A}{2} > 0.\] 

Since the triangle $-1\zeta\zeta'$ is counterclockwise oriented, we must have $0 < B - C < \pi$. Therefore \[\cos \dfrac{B - C}{2} > 0.\] The restrictions $0 < B - A < \pi$ and $0 < B - C < \pi$ imply that we can pick $B - \pi < A, C < B$. It is easy to see from the cusp diagram that $\vec{a} + \vec{b} + \vec{c} = 2$, so we must have 
\begin{equation}\label{Eq:CuspAngles}
-\pi < \min{A, C} < B <\pi, \quad \text{where} \quad B - \pi < A, C < B,
\end{equation} 
for otherwise if $A, C > 0$, then $\arg(\vec{a} + \vec{b} + \vec{c}) > 0$, which is a contradiction. Therefore $Z>0$ if and only if \[ -\pi < \dfrac{A + C}{2} < 0. \] The result follows.
\end{proof}

To prove geometric canonicity at the face $(\zeta\zeta'\infty)$, we require the notion of \emph{handedness}. Since handedness is only employed a few times, we keep the discussion brief. The reader is referred to Section~3 of Gu\'{e}ritaud--Schleimer \cite{GueritaudSchleimer}. 

\begin{definition} 
The \emph{handedness} of $g \in \text{GL}_2(\mathbb{C})$ is defined as follows
\[ \text{hand}(g) = \dfrac{(\text{tr}\,g)^2}{\text{Det}\,g}. \]
\end{definition}

When a loxodromy of $\mathbb{H}^3$ is conjugate to $z \mapsto \alpha z$ with $\abs{\alpha} > 1$ and $\text{Im}(\alpha) > 0$, we say it is \emph{left-handed}. One can easily verify that the M\"{o}bius transformation corresponding to $g$ is left-handed if and only if $\text{Im}(\text{hand}(g))$ is positive. Gu\'{e}ritaud and Schleimer prove the following result in \cite[Proposition~23]{GueritaudSchleimer}.

\begin{lemma}\label{Lem:MobLeft}
The M\"{o}bius transformation defined earlier by \[ f: u \mapsto \dfrac{\vec{a}\vec{c}}{u + 1} \] is left-handed.
\end{lemma}

\begin{lemma}\label{Lem:AC}
The angles $A, C$ of the vectors $\vec{a}, \vec{c}$ of the hexagon satisfy
\[ -\pi < \dfrac{A + C}{2} < 0. \]
\end{lemma}

\begin{proof}
Consider again the M\"{o}bius transformation $f: u \mapsto \vec{a}\vec{c}/(u + 1)$. By the definition of handedness, we have $\text{hand}(f) = 4/\vec{a}\vec{c}$. By \reflem{MobLeft}, we know that $f$ is left-handed, therefore \[\text{Im}(\text{hand}(g)) = \dfrac{-4\text{Im}(\vec{a}\vec{c})}{\abs{\vec{a}\vec{c}}^2} > 0,\] which implies that  $\text{Im}(\vec{a}\vec{c}) < 0$ or $A + C \in (-\pi, 0)$ modulo $2\pi$. But in the proof of \reflem{A+C/2}, we show that $-2\pi < A + C < \pi$. Hence $-\pi < A + C < 0$ and so \[ -\pi < \dfrac{A + C}{2} < 0. \qedhere\] 
\end{proof}

The next result follows immediately from \reflem{A+C/2} and \reflem{AC}.

\begin{lemma}\label{Lem:FaceHex}
Geometric canonicity holds at the face $(\zeta\zeta'\infty)$ of the hexagon.
\end{lemma}

\begin{proof}

To prove geometric canonicity at the face $(\zeta\zeta’\infty)$ of the hexagon, we need to check the convexity inequality of \refprop{LocalConvexity}. By \reflem{A+C/2}, the convexity inequality is satisfied if and only if the angles $A, C$ of the vectors $\vec{a}, \vec{c}$ of the hexagon  satisfy \[ -\pi < \dfrac{A + C}{2} < 0. \] But in \reflem{AC}, we showed that indeed this relationship holds. Geometric canonicity at the face $(\zeta\zeta’\infty)$ is proved. \qedhere

\end{proof}

\begin{lemma}\label{Lem:IntLST}
Geometric canonicity holds in the interior of a layered solid torus.
\end{lemma}

\begin{proof}
By \reflem{FaceHex}, geometric canonicity holds at the face $(\zeta\zeta’\infty)$ of the hexagon. But this face alternates every other one around the hexagon. Therefore, showing local convexity at one face automatically gives the result for the other faces. Moreover, the hexagon is symmetric, so opposite faces are identical. Thus an argument for one face immediately applies to the face on the opposite side. This proves geometric canonicity at every interior face in a layered solid torus. \qedhere

\end{proof}

Gu\'{e}ritaud and Schleimer prove the following result in \cite[Section~4.4]{GueritaudSchleimer}. The argument is similar in nature to the argument of \reflem{IntLST}, that is, we check the convexity inequality of \refprop{LocalConvexity} in Minkowski space.

\begin{lemma}\label{Lem:CoreLST}
Geometric canonicity holds at the core of a layered solid torus. 
\end{lemma}

\begin{proof}
Recall that the final step in the construction of a layered solid torus is performed by folding the ideal triangles of the innermost tetrahedron across an edge. This corresponds to a $3$-valent picture at the centre of the cusp diagram. In this case, we may label the vertices in the cusp picture as in \reffig{GSHex} but now we set $\zeta’ = -1$. Then the “hexagon” at the centre can be described by the broken line $(\zeta, -1, 1, -\zeta)$.

To prove geometric canonicity at the face $(-1 1 \infty)$, we must first compute the horoballs centred at the vertices $\zeta, -\zeta, \infty, -1, 1$. Note that this is done in exactly the same way as in \reflem{diamac} and \reflem{diamab}. Using equation~\eqref{Eq:IsoVec}, the isotropic vectors corresponding to these horoballs are respectively:

\[ \begin{array}{ccccccccc}

\vspace{2mm}

v_\zeta & = & \frac{1}{\abs{\zeta + 1}^2} & ( & 2\xi, & 2\eta, & 1 - \xi^2 - \eta^2, & 1 + \xi^2 + \eta^2 & ) \\ \vspace{2mm}

v_{-\zeta} & = & \frac{1}{\abs{\zeta + 1}^2} & ( & -2\xi, & -2\eta, & 1 - \xi^2 - \eta^2, & 1 + \xi^2 + \eta^2 & ) \\ \vspace{2mm}

v_\infty & = & & ( & 0, & 0, & -1, & 1 & ) \\ \vspace{2mm}

v_{-1} & = & \frac{1}{2 \abs{\zeta + 1}} & ( & -2, &  0, &  0, & 2 & ) \\ \vspace{2mm}

v_{1} & = & \frac{1}{2 \abs{\zeta + 1}} & ( & 2, & 0, & 0, & 2 & ). \\

\end{array} \] 

Solving the equation $\rho v_{\zeta} + (1 - \rho) v_{-\zeta} = \lambda v_{\infty} + \mu v_{1} + \nu v_{-1}$ gives a unique solution $\rho=1/2$, \[ \lambda = \dfrac{\abs{\zeta}^2 - 1}{\abs{\zeta + 1}^2} \quad \text{and} \quad \mu = \nu = \dfrac{1}{\abs{\zeta + 1}}. \] Then \[\lambda + \mu + \nu = \dfrac{\abs{\zeta}^2 - 1 + 2 \abs{\zeta + 1}}{\abs{\zeta + 1}^2}. \]

To check geometric canonicity, we need to show that $\lambda + \mu + \nu > 1$, or equivalently $\abs{\zeta}^2 > (\abs{\zeta + 1} - 1)^2$. But the triangle inequality for the Euclidean triangle $(0, -1, \zeta)$ asserts that $\abs{\zeta} + 1 > \abs{\zeta + 1}$. Hence  $\lambda + \mu + \nu > 1$; the convexity inequality of \refprop{LocalConvexity} holds. \qedhere

\end{proof}

The next lemma is a direct consequence of \reflem{IntLST} and \reflem{CoreLST}.

\begin{lemma}\label{Lem:GeomCanDLST}
Geometric canonicity holds in the interior and at the core of the double cover of a layered solid torus. 
\end{lemma}

\begin{proof}
  The double cover of a canonical triangulation gives a canonical triangulation. \qedhere
\end{proof}

The following result is proved in Section~5.3 of Gu\'{e}ritaud--Schleimer \cite{GueritaudSchleimer}. The proof is again by checking the convexity statement of \refprop{LocalConvexity}.

\begin{lemma}\label{Lem:GeomCanSideLST} 
Geometric canonicity holds in the interior and at the core of a side-by-side layered solid torus. 
\end{lemma}

\begin{proof}

The interior faces of a side-by-side layered solid torus are treated exactly as in \reflem{FaceHex}, except that now we require a result analogous to \reflem{MobLeft} for a side-by-side layered solid torus. Using the notion of handedness, Gu\'{e}ritaud and Schleimer prove such a result; see \cite[Proposition~32]{GueritaudSchleimer}. We omit the proof here as we have kept the discussion on handedness in this paper to a minimum. Nevertheless, this result allows us to compute the radii of the horoballs in the cusp diagram of a side-by-side layered solid torus in exactly the same way as in \reflem{diamac} and \reflem{diamab}. The argument of \reflem{FaceHex} now follows through unchanged. 

It remains to prove geometric canonicity at the core faces of a side-by-side layered solid torus. Recall that the final step in the construction of a side-by-side layered solid torus involves identifying an edge across the last pair of tetrahedra in the layering. This corresponds to identifying a pair of opposite vertices of the innermost hexagon in the cusp diagram. The result is a $4$-valent picture at the centre. In this case, we may label the vertices of the innermost hexagon as $-1, 0, \zeta, 1, 0, -\zeta$. 

We shall prove geometric canonicity at the face $(0 \zeta \infty)$. As mentioned above, we compute the horoballs centred at $-1, 1, \infty, 0, \zeta$ as in \reflem{diamac} and \reflem{diamab}. Using equation~\eqref{Eq:IsoVec}, the isotropic vectors corresponding to the above horoballs are respectively:

\[ \begin{array}{ccccccccc}

\vspace{2mm}

v_{-1} & = & \frac{1}{\abs{\zeta - 1}} & ( & -2, &  0, &  0, & 2 & ) \\ \vspace{2mm}

v_{1} & = & \frac{1}{\abs{\zeta - 1}} & ( & 2, & 0, & 0, & 2 & ) \\ \vspace{2mm}

v_\infty & = & & ( & 0, & 0, & -1, & 1 & ) \\ \vspace{2mm}

v_0 & = & \frac{1}{\abs{\zeta}} & ( & 0, & 0, & 1, & 1 & ) \\ \vspace{2mm}

v_\zeta & = & \frac{1}{\abs{\zeta} \abs{\zeta - 1}} & ( & 2\xi, & 2\eta, & 1 - \abs{\zeta}^2, & 1 + \abs{\zeta}^2 & ). \\

\end{array} \]

The equation $\lambda v_{\infty} + \mu v_0 + \nu v_{\zeta} = \rho v_1 + (1 - \rho) v_{-1}$ can be solved to obtain a unique solution:

\[ \lambda = \dfrac{1}{\abs{\zeta - 1}}, \quad \mu = \dfrac{\abs{\zeta}}{\abs{\zeta - 1}}, \quad \text{and} \quad \nu = 0. \] 

Therefore \[ \lambda + \mu + \nu = \dfrac{\abs{\zeta} + 1}{\abs{\zeta - 1}} > 1, \] because the triangle inequality for the triangle $(0, 1, \zeta)$ asserts that $\abs{\zeta} + 1 > \abs{\zeta - 1}$. The convexity inequality of \refprop{LocalConvexity} is satisfied, so geometric canonicity holds at the face $(0 \zeta \infty)$. \qedhere

\end{proof}

Let $M$ be the manifold obtained by Dehn filling the crossing circles cusps of the Borromean rings link complement along slopes $m_1, m_2 \in (\QQ\cup\{1/0\})- \{0, 1/0, \pm 1, \pm 2\}$. Let $K$ be the unfilled cusp of $M$. We now show that geometric canonicity holds at a face on the boundary of the Borromean rings link complement. By \reflem{OneFace} and \reflem{ThreeCases}, there are three cases to consider.

{\bf Case 1: One positive slope and one negative slope}

First note that this is the case on the left of \reffig{Borr3Cases}. Label the vertices in the cusp diagram of $K$ as in \reffig{Borr3Cases} (left). Notice that the triangle inside the non-convex hexagon can be considered as another layer on the solid torus with cusp triangulation given by the convex hexagon. Then the argument of \reflem{FaceHex} follows through unchanged.

{\bf Case 2: Two negative slopes}

First note that this is the middle case in \reffig{Borr3Cases}. Label the vertices in the cusp diagram of $K$ as in \reffig{Borr3Cases} (middle). The argument of \reflem{FaceHex} follows through up to and including \reflem{A+C/2}. Thus, to prove geometric canonicity at a face on the boundary of the Borromean rings link complement in this case, it remains to show that \[ -\pi < \dfrac{A + C}{2} < 0. \] 

\begin{lemma}
Suppose that $m_1$ and $m_2$ are both negative and within the allowed set. Then the angles $A, C$ of the vectors $\vec{a}, \vec{c}$ in the cusp diagram of $K$ must satisfy one of the following three cases.
\begin{itemize}
\item[(1)] $A > 0$ and $C < 0$
\item[(2)] $A < 0$ and $C > 0$
\item[(3)] $A < 0$ and $C < 0$
\end{itemize}
In case~(1), the vertices $-1, 1,\zeta,\zeta'$ have the configuration of \reffig{Case2Cases} (left). In case~(2), these vertices have the configuration of \reffig{Case2Cases} (middle), and in case~(3), these vertices have the configuration of \reffig{Case2Cases} (right). 
\end{lemma}

\begin{proof}
Rotate the diagram of \reffig{Borr3Cases} (middle) so that $-1$ and $1$ actually lie on the real axis. Then $A$ is the angle that $\vec{a}$ makes with the real axis, where $\vec{a}$ has endpoints at $-1$ and $\zeta$.
  
Recall that the triangles $-1\zeta\zeta'$ and $1\zeta'\zeta$ are oriented anticlockwise. Next consider the triangles $-1\zeta1$ and $-1\zeta'1$. Note that in \reffig{Borr3Cases} (middle), one of these is oriented anticlockwise and the other clockwise, but it could be that both are oriented clockwise. If so, then $\zeta$ must lie inside $-1\zeta'1$ as in \reffig{Case2Cases} (left), for otherwise $-1\zeta\zeta'$ is oriented clockwise. In this case, we must have $A > 0$ and $C < 0$.

\begin{figure}
\centering
  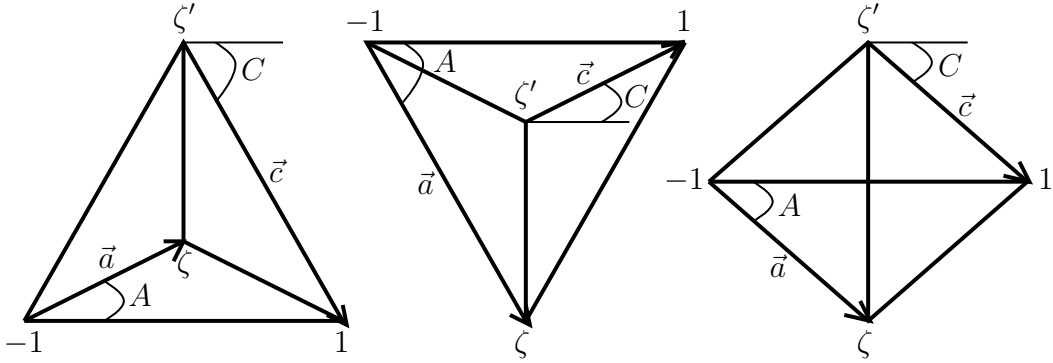
  \caption{Possible configurations of $-1, 1, \zeta, \zeta'$ in case 2.}
  \label{Fig:Case2Cases}
\end{figure}

If both of the triangles $-1\zeta1$ and $-1\zeta'1$ are oriented anticlockwise, then $\zeta'$ must lie inside $-1\zeta1$, as in \reffig{Case2Cases} (middle), for otherwise $-1\zeta\zeta'$ is oriented clockwise. In this case, we have $A < 0$ and $C > 0$.

If $-1\zeta1$ is clockwise and $-1\zeta'1$ is anticlockwise, then $-1\zeta\zeta'$ and $1\zeta'\zeta$ are clockwise, which is a contradiction. So $-1\zeta1$ is anticlockwise and $-1\zeta'1$ is clockwise, which forces $A < 0$ and $C < 0$, as in \reffig{Case2Cases} (right).
\end{proof}

\begin{lemma}\label{Lem:AngNeg}
Suppose that $m_1$ and $m_2$ are both negative and within the allowed set. Then the angles $A, C$ of the vectors $\vec{a}, \vec{c}$ in the cusp diagram of $K$ must satisfy $A < 0$ and $C < 0$.
\end{lemma}

\begin{proof}
Suppose for a contradiction that $A > 0$ and $C < 0$. In the proof of the previous lemma, this corresponds to $-1, 1,\zeta,\zeta'$ having the configuration of \reffig{Case2Cases} (left). Consider the edges of the two adjacent hexagons in the cusp; see again \reffig{Borr3Cases} (middle). Notice that the edge from $-1$ to $\zeta'$ runs along the vector $\vec{a} + \vec{b}$. The edge from $\zeta$ to $\zeta'$ runs along the vector $\vec{b}$.

The edge from $\zeta$ to $1$ runs along the vector $\vec{b} + \vec{c}$. This is an edge of the convex hexagon in \reffig{Borr3Cases} (middle), but it has the same length and direction as an edge of the non-convex hexagon with an endpoint at $\zeta$. Let the other endpoint of that edge be denoted $\lambda$, as in \reffig{Borr3Cases} (middle). Then the edge from $\lambda$ to $\zeta$ runs along the vector $\vec{b} + \vec{c}$.

%% Now observe that the edge from $\zeta$ to $1$ is also along the vector $\vec{b} + \vec{c}$. This means that the edge from $\zeta$ to $\lambda$ lies on the same line as that from $1$ to $\zeta$ with the same length. 

But then, in the case that  $A > 0$ and $C < 0$, the edge from $\zeta$ to $\lambda$ runs into the triangle $-1\zeta\zeta'$; see again \reffig{Case2Cases} (left). That is, the boundary of the convex hexagon intersects the triangle $-1\zeta\zeta'$. But this cannot happen because the triangle $-1\zeta\zeta'$ is geometric and therefore must be contained in the interior of the convex hexagon. Hence the case that $A > 0$ and $C < 0$ cannot happen.

Suppose for a contradiction that $A < 0$ and $C > 0$, then the vertices $-1, 1, \zeta, \zeta'$ have the configuration of \reffig{Case2Cases} (middle). Consider again the edges of the two adjacent hexagons in the cusp; see \reffig{Borr3Cases} (middle). Notice that the edge from $\zeta$ to $1$ runs along the vector $\vec{b} + \vec{c}$. The edge from $\zeta$ to $\zeta'$ runs along the vector $\vec{b}$. Moreover, the edge from $-1$ to $\zeta'$ runs along the vector $\vec{a} + \vec{b}$. Again observe that this side of the non-convex hexagon is identical in length and direction to a side of the convex hexagon with endpoint at $\zeta'$; we define the other endpoint to be $\lambda'$ as in \reffig{Borr3Cases} (middle). Thus the edge from $-1$ to $\zeta'$ also runs along the vector $\vec{a} + \vec{b}$. But then, in the case that $A < 0$ and $C > 0$, the edge from $\zeta'$ to $\lambda'$ intersects the triangle $1\zeta'\zeta$; see \reffig{Case2Cases} (middle). That is, the boundary of the non-convex hexagon intersects the triangle $1\zeta'\zeta$. This is a contraction because the geometric triangle $1\zeta'\zeta$ lies in the interior of the non-convex hexagon. Hence the case that $A < 0$ and $C > 0$ cannot happen. 
\end{proof}

By \reflem{AngNeg}, we have $A < 0$ and $C < 0$, therefore $A + C < 0$. By equation~\eqref{Eq:CuspAngles}, we have $2B - 2\pi < A + C$, which implies that $A + C > -2\pi$ because $B > 0$. Geometry canonicity holds at a face on the boundary of the Borromean rings link complement in this case.

{\bf Case 3: Two positive slopes}

Label the vertices in the cusp diagram of $K$ as in \reffig{Borr3Cases} (right). The argument of \reflem{FaceHex} follows through up to and including \reflem{A+C/2}. Thus, it remains to show that \[ -\pi < \dfrac{A + C}{2} < 0. \] Consider again the edges of the two adjacent hexagons in the cusp; see \reffig{Borr3Cases} (right). The edge from $-1$ to $\zeta$ runs along the vector $\vec{a}$. Observe that this side of the non-convex hexagon is identical in length and direction to a side of the convex hexagon with endpoints at $\zeta$ and $1$. Thus the edge from $\zeta$ to $1$ also runs along the vector $\vec{a}$. Rotate the diagram of \reffig{Borr3Cases} (right) so that $-1$ and $1$ actually lie on the real axis; see \reffig{Case3BR}. Then the angle that $\vec{a}$ makes with the real axis is zero, that is, $A = 0$.

Also, we must have $C < 0$ because the triangles $-1\zeta\zeta'$ and $1\zeta'\zeta$ are oriented anticlockwise. So $A + C < 0$. By equation~\eqref{Eq:CuspAngles}, we have $2B - 2\pi < A + C$, therefore $A + C > -2\pi$ since $B > 0$. Geometric canonicity holds in this case.

\begin{figure}
  \centering
  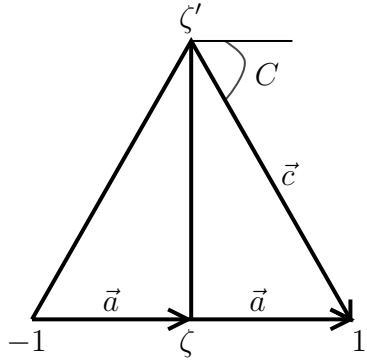
  \caption{The only possibility in case 3.}
  \label{Fig:Case3BR}
\end{figure}

We have now proved the following result.

\begin{lemma}\label{Lem:GeomCanBoundBR}
Geometric canonicity holds on the boundary of the Borromean rings link complement.
\end{lemma}

\begin{theorem}\label{Thm:CanonBorromean}
Let $L$ be a fully augmented link with exactly two crossing circles, as in \reffig{BorromeanRings}. Let $M$ be the manifold obtained by Dehn filling the crossing circles of $S^3-L$ along slopes $m_1, m_2 \in (\QQ\cup\{1/0\})- \{0, 1/0, \pm 1, \pm 2\}$.
Then $M$ admits a canonical triangulation.
\end{theorem}

\begin{proof}
Follows directly from \reflem{GeomCanBoundBR}, \reflem{GeomCanDLST}, and \reflem{GeomCanSideLST}. 
\end{proof}

\bibliographystyle{amsplain}
\bibliography{biblio}

\end{document}

%% file: figures/BorrCirclePacking.pdf_tex
%% Creator: Inkscape 1.0.2-2 (e86c870879, 2021-01-15), www.inkscape.org
%% Accompanies image file 'BorrCirclePacking.pdf' (pdf, eps, ps)
\begingroup%
  \makeatletter%
  \providecommand\color[2][]{%
    \errmessage{(Inkscape) Color is used for the text in Inkscape, but the package 'color.sty' is not loaded}%
    \renewcommand\color[2][]{}%
  }%
  \providecommand\transparent[1]{%
    \errmessage{(Inkscape) Transparency is used (non-zero) for the text in Inkscape, but the package 'transparent.sty' is not loaded}%
    \renewcommand\transparent[1]{}%
  }%
  \providecommand\rotatebox[2]{#2}%
  \newcommand*\fsize{\dimexpr\f@size pt\relax}%
  \newcommand*\lineheight[1]{\fontsize{\fsize}{#1\fsize}\selectfont}%
  \ifx\svgwidth\undefined%
    \setlength{\unitlength}{330.27041245bp}%
    \ifx\svgscale\undefined%
      \relax%
    \else%
      \setlength{\unitlength}{\unitlength * \real{\svgscale}}%
    \fi%
  \else%
    \setlength{\unitlength}{\svgwidth}%
  \fi%
  \global\let\svgwidth\undefined%
  \global\let\svgscale\undefined%
  \makeatother%
  \begin{picture}(1,0.23544991)%
    \lineheight{1}%
    \setlength\tabcolsep{0pt}%
    \put(0,0){\includegraphics[width=\unitlength,page=1]{BorrCirclePacking.pdf}}%
    \put(0.51024774,0.00821805){\color[rgb]{0,0,0}\makebox(0,0)[lt]{\lineheight{1.25}\smash{\begin{tabular}[t]{l}$A$\end{tabular}}}}%
    \put(0.56521059,0.00814419){\color[rgb]{0,0,0}\makebox(0,0)[lt]{\lineheight{1.25}\smash{\begin{tabular}[t]{l}$B$\end{tabular}}}}%
    \put(0.5422832,0.18922172){\color[rgb]{0,0,0}\makebox(0,0)[lt]{\lineheight{1.25}\smash{\begin{tabular}[t]{l}$C$\end{tabular}}}}%
    \put(0.54471627,0.09627873){\color[rgb]{0,0,0}\makebox(0,0)[lt]{\lineheight{1.25}\smash{\begin{tabular}[t]{l}$D$\end{tabular}}}}%
    \put(0,0){\includegraphics[width=\unitlength,page=2]{BorrCirclePacking.pdf}}%
    \put(0.11320175,0.21154151){\color[rgb]{0,0,0}\makebox(0,0)[lt]{\lineheight{1.25}\smash{\begin{tabular}[t]{l}$A$\end{tabular}}}}%
    \put(0.05937552,0.16619368){\color[rgb]{0,0,0}\makebox(0,0)[lt]{\lineheight{1.25}\smash{\begin{tabular}[t]{l}$B$\end{tabular}}}}%
    \put(0.02492736,0.09873979){\color[rgb]{0,0,0}\makebox(0,0)[lt]{\lineheight{1.25}\smash{\begin{tabular}[t]{l}$C$\end{tabular}}}}%
    \put(0.102974,0.10539115){\color[rgb]{0,0,0}\makebox(0,0)[lt]{\lineheight{1.25}\smash{\begin{tabular}[t]{l}$D$\end{tabular}}}}%
    \put(0.60324066,0.00849124){\color[rgb]{0,0,0}\makebox(0,0)[lt]{\lineheight{1.25}\smash{\begin{tabular}[t]{l}$B'$\end{tabular}}}}%
    \put(0.6538113,0.00834796){\color[rgb]{0,0,0}\makebox(0,0)[lt]{\lineheight{1.25}\smash{\begin{tabular}[t]{l}$A'$\end{tabular}}}}%
    \put(0.63399841,0.19065734){\color[rgb]{0,0,0}\makebox(0,0)[lt]{\lineheight{1.25}\smash{\begin{tabular}[t]{l}$C'$\end{tabular}}}}%
    \put(0.63636799,0.0965108){\color[rgb]{0,0,0}\makebox(0,0)[lt]{\lineheight{1.25}\smash{\begin{tabular}[t]{l}$D'$\end{tabular}}}}%
    \put(0,0){\includegraphics[width=\unitlength,page=3]{BorrCirclePacking.pdf}}%
    \put(0.82573602,0.19053686){\color[rgb]{0,0,0}\makebox(0,0)[lt]{\lineheight{1.25}\smash{\begin{tabular}[t]{l}$A$\end{tabular}}}}%
    \put(0.82573602,0.09506182){\color[rgb]{0,0,0}\makebox(0,0)[lt]{\lineheight{1.25}\smash{\begin{tabular}[t]{l}$B$\end{tabular}}}}%
    \put(0.79639575,0.00850067){\color[rgb]{0,0,0}\makebox(0,0)[lt]{\lineheight{1.25}\smash{\begin{tabular}[t]{l}$C$\end{tabular}}}}%
    \put(0.8543651,0.00845169){\color[rgb]{0,0,0}\makebox(0,0)[lt]{\lineheight{1.25}\smash{\begin{tabular}[t]{l}$D$\end{tabular}}}}%
    \put(0,0){\includegraphics[width=\unitlength,page=4]{BorrCirclePacking.pdf}}%
    \put(0.91682012,0.0956048){\color[rgb]{0,0,0}\makebox(0,0)[lt]{\lineheight{1.25}\smash{\begin{tabular}[t]{l}$B'$\end{tabular}}}}%
    \put(0.9172561,0.19026864){\color[rgb]{0,0,0}\makebox(0,0)[lt]{\lineheight{1.25}\smash{\begin{tabular}[t]{l}$A'$\end{tabular}}}}%
    \put(0.9387812,0.00826546){\color[rgb]{0,0,0}\makebox(0,0)[lt]{\lineheight{1.25}\smash{\begin{tabular}[t]{l}$C'$\end{tabular}}}}%
    \put(0.89143474,0.00842157){\color[rgb]{0,0,0}\makebox(0,0)[lt]{\lineheight{1.25}\smash{\begin{tabular}[t]{l}$D'$\end{tabular}}}}%
    \put(0,0){\includegraphics[width=\unitlength,page=5]{BorrCirclePacking.pdf}}%
    \put(0.31449781,0.21073059){\color[rgb]{0,0,0}\makebox(0,0)[lt]{\lineheight{1.25}\smash{\begin{tabular}[t]{l}$A'$\end{tabular}}}}%
    \put(0.26067159,0.16538276){\color[rgb]{0,0,0}\makebox(0,0)[lt]{\lineheight{1.25}\smash{\begin{tabular}[t]{l}$B'$\end{tabular}}}}%
    \put(0.22622341,0.09792887){\color[rgb]{0,0,0}\makebox(0,0)[lt]{\lineheight{1.25}\smash{\begin{tabular}[t]{l}$C'$\end{tabular}}}}%
    \put(0.30427005,0.10458023){\color[rgb]{0,0,0}\makebox(0,0)[lt]{\lineheight{1.25}\smash{\begin{tabular}[t]{l}$D'$\end{tabular}}}}%
    \put(0,0){\includegraphics[width=\unitlength,page=6]{BorrCirclePacking.pdf}}%
  \end{picture}%
\endgroup%

%% file: figures/BorroGluing.pdf_tex
%% Creator: Inkscape 1.0.2-2 (e86c870879, 2021-01-15), www.inkscape.org
%% Accompanies image file 'BorroGluing.pdf' (pdf, eps, ps)
\begingroup%
  \makeatletter%
  \providecommand\color[2][]{%
    \errmessage{(Inkscape) Color is used for the text in Inkscape, but the package 'color.sty' is not loaded}%
    \renewcommand\color[2][]{}%
  }%
  \providecommand\transparent[1]{%
    \errmessage{(Inkscape) Transparency is used (non-zero) for the text in Inkscape, but the package 'transparent.sty' is not loaded}%
    \renewcommand\transparent[1]{}%
  }%
  \providecommand\rotatebox[2]{#2}%
  \newcommand*\fsize{\dimexpr\f@size pt\relax}%
  \newcommand*\lineheight[1]{\fontsize{\fsize}{#1\fsize}\selectfont}%
  \ifx\svgwidth\undefined%
    \setlength{\unitlength}{244.95000458bp}%
    \ifx\svgscale\undefined%
      \relax%
    \else%
      \setlength{\unitlength}{\unitlength * \real{\svgscale}}%
    \fi%
  \else%
    \setlength{\unitlength}{\svgwidth}%
  \fi%
  \global\let\svgwidth\undefined%
  \global\let\svgscale\undefined%
  \makeatother%
  \begin{picture}(1,0.21922824)%
    \lineheight{1}%
    \setlength\tabcolsep{0pt}%
    \put(0,0){\includegraphics[width=\unitlength,page=1]{BorroGluing.pdf}}%
    \put(0.13120104,0.15161604){\color[rgb]{0,0,0}\makebox(0,0)[lt]{\lineheight{1.25}\smash{\begin{tabular}[t]{l}$a$\end{tabular}}}}%
    \put(0.62381304,0.0301563){\color[rgb]{0,0,0}\makebox(0,0)[lt]{\lineheight{1.25}\smash{\begin{tabular}[t]{l}$a'$\end{tabular}}}}%
    \put(0.0630943,0.03176457){\color[rgb]{0,0,0}\makebox(0,0)[lt]{\lineheight{1.25}\smash{\begin{tabular}[t]{l}$b$\end{tabular}}}}%
    \put(0.27669333,0.15464679){\color[rgb]{0,0,0}\makebox(0,0)[lt]{\lineheight{1.25}\smash{\begin{tabular}[t]{l}$c$\end{tabular}}}}%
    \put(0.69281737,0.1546462){\color[rgb]{0,0,0}\makebox(0,0)[lt]{\lineheight{1.25}\smash{\begin{tabular}[t]{l}$b'$\end{tabular}}}}%
    \put(0.89617789,0.03086715){\color[rgb]{0,0,0}\makebox(0,0)[lt]{\lineheight{1.25}\smash{\begin{tabular}[t]{l}$c'$\end{tabular}}}}%
    \put(0.33094401,0.02944506){\color[rgb]{0,0,0}\makebox(0,0)[lt]{\lineheight{1.25}\smash{\begin{tabular}[t]{l}$d$\end{tabular}}}}%
    \put(0.84019419,0.15470899){\color[rgb]{0,0,0}\makebox(0,0)[lt]{\lineheight{1.25}\smash{\begin{tabular}[t]{l}$d'$\end{tabular}}}}%
  \end{picture}%
\endgroup%

%% file: figures/LLFarey.pdf_tex
%% Creator: Inkscape 1.0.2-2 (e86c870879, 2021-01-15), www.inkscape.org
%% Accompanies image file 'LLFarey.pdf' (pdf, eps, ps)
\begingroup%
  \makeatletter%
  \providecommand\color[2][]{%
    \errmessage{(Inkscape) Color is used for the text in Inkscape, but the package 'color.sty' is not loaded}%
    \renewcommand\color[2][]{}%
  }%
  \providecommand\transparent[1]{%
    \errmessage{(Inkscape) Transparency is used (non-zero) for the text in Inkscape, but the package 'transparent.sty' is not loaded}%
    \renewcommand\transparent[1]{}%
  }%
  \providecommand\rotatebox[2]{#2}%
  \newcommand*\fsize{\dimexpr\f@size pt\relax}%
  \newcommand*\lineheight[1]{\fontsize{\fsize}{#1\fsize}\selectfont}%
  \ifx\svgwidth\undefined%
    \setlength{\unitlength}{218.53619385bp}%
    \ifx\svgscale\undefined%
      \relax%
    \else%
      \setlength{\unitlength}{\unitlength * \real{\svgscale}}%
    \fi%
  \else%
    \setlength{\unitlength}{\svgwidth}%
  \fi%
  \global\let\svgwidth\undefined%
  \global\let\svgscale\undefined%
  \makeatother%
  \begin{picture}(1,0.70364905)%
    \lineheight{1}%
    \setlength\tabcolsep{0pt}%
    \put(0,0){\includegraphics[width=\unitlength,page=1]{LLFarey.pdf}}%
    \put(0.45033123,0.68453215){\color[rgb]{0,0,0}\makebox(0,0)[lt]{\lineheight{1.25}\smash{\begin{tabular}[t]{l}$1/1$\end{tabular}}}}%
    \put(0.82282233,0.34051275){\color[rgb]{0,0,0}\makebox(0,0)[lt]{\lineheight{1.25}\smash{\begin{tabular}[t]{l}$0/1$\end{tabular}}}}%
    \put(0.19321186,0.5785729){\color[rgb]{0,0,0}\makebox(0,0)[lt]{\lineheight{1.25}\smash{\begin{tabular}[t]{l}$2/1$\end{tabular}}}}%
    \put(0.7177594,0.58364048){\color[rgb]{0,0,0}\makebox(0,0)[lt]{\lineheight{1.25}\smash{\begin{tabular}[t]{l}$1/2$\end{tabular}}}}%
    \put(0,0){\includegraphics[width=\unitlength,page=2]{LLFarey.pdf}}%
    \put(-0.04034457,0.33726008){\color[rgb]{0,0,0}\makebox(0,0)[lt]{\lineheight{1.25}\smash{\begin{tabular}[t]{l}$1/0=\infty$\end{tabular}}}}%
    \put(0.43883395,-0.0198624){\color[rgb]{0,0,0}\makebox(0,0)[lt]{\lineheight{1.25}\smash{\begin{tabular}[t]{l}$-1/1$\end{tabular}}}}%
  \end{picture}%
\endgroup%

%% file: figures/TriangleFolding.pdf_tex
%% Creator: Inkscape 1.0.2-2 (e86c870879, 2021-01-15), www.inkscape.org
%% Accompanies image file 'TriangleFolding.pdf' (pdf, eps, ps)
\begingroup%
  \makeatletter%
  \providecommand\color[2][]{%
    \errmessage{(Inkscape) Color is used for the text in Inkscape, but the package 'color.sty' is not loaded}%
    \renewcommand\color[2][]{}%
  }%
  \providecommand\transparent[1]{%
    \errmessage{(Inkscape) Transparency is used (non-zero) for the text in Inkscape, but the package 'transparent.sty' is not loaded}%
    \renewcommand\transparent[1]{}%
  }%
  \providecommand\rotatebox[2]{#2}%
  \newcommand*\fsize{\dimexpr\f@size pt\relax}%
  \newcommand*\lineheight[1]{\fontsize{\fsize}{#1\fsize}\selectfont}%
  \ifx\svgwidth\undefined%
    \setlength{\unitlength}{270.34574318bp}%
    \ifx\svgscale\undefined%
      \relax%
    \else%
      \setlength{\unitlength}{\unitlength * \real{\svgscale}}%
    \fi%
  \else%
    \setlength{\unitlength}{\svgwidth}%
  \fi%
  \global\let\svgwidth\undefined%
  \global\let\svgscale\undefined%
  \makeatother%
  \begin{picture}(1,0.38813344)%
    \lineheight{1}%
    \setlength\tabcolsep{0pt}%
    \put(0,0){\includegraphics[width=\unitlength,page=1]{TriangleFolding.pdf}}%
    \put(0.29809658,0.15048187){\color[rgb]{0,0,0}\makebox(0,0)[lt]{\lineheight{0}\smash{\begin{tabular}[t]{l}$m$\end{tabular}}}}%
    \put(0.14258338,0.15101219){\color[rgb]{0,0,0}\makebox(0,0)[lt]{\lineheight{0}\smash{\begin{tabular}[t]{l}$m'$\end{tabular}}}}%
    \put(0.06752692,0.21416002){\color[rgb]{0,0,0}\makebox(0,0)[lt]{\lineheight{0}\smash{\begin{tabular}[t]{l}$s$\end{tabular}}}}%
    \put(0.3933985,0.21118802){\color[rgb]{0,0,0}\makebox(0,0)[lt]{\lineheight{0}\smash{\begin{tabular}[t]{l}$s$\end{tabular}}}}%
    \put(0.22506407,0.02135136){\color[rgb]{0,0,0}\makebox(0,0)[lt]{\lineheight{0}\smash{\begin{tabular}[t]{l}$t$\end{tabular}}}}%
    \put(0.23070963,0.3494035){\color[rgb]{0,0,0}\makebox(0,0)[lt]{\lineheight{0}\smash{\begin{tabular}[t]{l}$t$\end{tabular}}}}%
    \put(0,0){\includegraphics[width=\unitlength,page=2]{TriangleFolding.pdf}}%
    \put(0.76783798,0.26445877){\color[rgb]{0,0,0}\makebox(0,0)[lt]{\lineheight{0}\smash{\begin{tabular}[t]{l}$m$\end{tabular}}}}%
    \put(0.71669429,0.09211138){\color[rgb]{0,0,0}\makebox(0,0)[lt]{\lineheight{0}\smash{\begin{tabular}[t]{l}$m'$\end{tabular}}}}%
    \put(0.67045486,0.17338074){\color[rgb]{0,0,0}\makebox(0,0)[lt]{\lineheight{0}\smash{\begin{tabular}[t]{l}$t$\end{tabular}}}}%
    \put(0.76573614,0.36324263){\color[rgb]{0,0,0}\makebox(0,0)[lt]{\lineheight{0}\smash{\begin{tabular}[t]{l}$s$\end{tabular}}}}%
    \put(0.58708375,0.17338074){\color[rgb]{0,0,0}\makebox(0,0)[lt]{\lineheight{0}\smash{\begin{tabular}[t]{l}$s$\end{tabular}}}}%
  \end{picture}%
\endgroup%

%% file: figures/Borr3Cases.pdf_tex
%% Creator: Inkscape 1.0.2-2 (e86c870879, 2021-01-15), www.inkscape.org
%% Accompanies image file 'Borr3Cases.pdf' (pdf, eps, ps)
\begingroup%
  \makeatletter%
  \providecommand\color[2][]{%
    \errmessage{(Inkscape) Color is used for the text in Inkscape, but the package 'color.sty' is not loaded}%
    \renewcommand\color[2][]{}%
  }%
  \providecommand\transparent[1]{%
    \errmessage{(Inkscape) Transparency is used (non-zero) for the text in Inkscape, but the package 'transparent.sty' is not loaded}%
    \renewcommand\transparent[1]{}%
  }%
  \providecommand\rotatebox[2]{#2}%
  \newcommand*\fsize{\dimexpr\f@size pt\relax}%
  \newcommand*\lineheight[1]{\fontsize{\fsize}{#1\fsize}\selectfont}%
  \ifx\svgwidth\undefined%
    \setlength{\unitlength}{390.83646038bp}%
    \ifx\svgscale\undefined%
      \relax%
    \else%
      \setlength{\unitlength}{\unitlength * \real{\svgscale}}%
    \fi%
  \else%
    \setlength{\unitlength}{\svgwidth}%
  \fi%
  \global\let\svgwidth\undefined%
  \global\let\svgscale\undefined%
  \makeatother%
  \begin{picture}(1,0.49177241)%
    \lineheight{1}%
    \setlength\tabcolsep{0pt}%
    \put(0,0){\includegraphics[width=\unitlength,page=1]{Borr3Cases.pdf}}%
    \put(0.55249134,0.29595517){\color[rgb]{0,0,0}\makebox(0,0)[lt]{\lineheight{1.25}\smash{\begin{tabular}[t]{l}$\vec{c}$\end{tabular}}}}%
    \put(0.37159525,0.46946793){\color[rgb]{0,0,0}\makebox(0,0)[lt]{\lineheight{1.25}\smash{\begin{tabular}[t]{l}$1$\end{tabular}}}}%
    \put(0.53671898,0.16839712){\color[rgb]{0,0,0}\makebox(0,0)[lt]{\lineheight{1.25}\smash{\begin{tabular}[t]{l}$\vec{a}$\end{tabular}}}}%
    \put(0.22910859,0.40904151){\color[rgb]{0,0,0}\makebox(0,0)[lt]{\lineheight{1.25}\smash{\begin{tabular}[t]{l}$\vec{b}$\end{tabular}}}}%
    \put(0.26364992,0.47070815){\color[rgb]{0,0,0}\makebox(0,0)[lt]{\lineheight{1.25}\smash{\begin{tabular}[t]{l}$\vec{c}$\end{tabular}}}}%
    \put(0.20224311,0.295109){\color[rgb]{0,0,0}\makebox(0,0)[lt]{\lineheight{1.25}\smash{\begin{tabular}[t]{l}$\vec{a}$\end{tabular}}}}%
    \put(0.17193506,0.47340739){\color[rgb]{0,0,0}\makebox(0,0)[lt]{\lineheight{1.25}\smash{\begin{tabular}[t]{l}$\zeta'$\end{tabular}}}}%
    \put(0.66632838,0.34415946){\color[rgb]{0,0,0}\makebox(0,0)[lt]{\lineheight{1.25}\smash{\begin{tabular}[t]{l}$1$\end{tabular}}}}%
    \put(0.18091034,0.20423787){\color[rgb]{0,0,0}\makebox(0,0)[lt]{\lineheight{1.25}\smash{\begin{tabular}[t]{l}$-1$\end{tabular}}}}%
    \put(0.21759962,0.34500682){\color[rgb]{0,0,0}\makebox(0,0)[lt]{\lineheight{1.25}\smash{\begin{tabular}[t]{l}$\zeta$\end{tabular}}}}%
    \put(0.44669159,0.22639145){\color[rgb]{0,0,0}\makebox(0,0)[lt]{\lineheight{1.25}\smash{\begin{tabular}[t]{l}$\zeta'$\end{tabular}}}}%
    \put(0.63012978,0.22554771){\color[rgb]{0,0,0}\makebox(0,0)[lt]{\lineheight{1.25}\smash{\begin{tabular}[t]{l}$\zeta$\end{tabular}}}}%
    \put(0.92280442,0.23322113){\color[rgb]{0,0,0}\makebox(0,0)[lt]{\lineheight{1.25}\smash{\begin{tabular}[t]{l}$\zeta'$\end{tabular}}}}%
    \put(0.74405491,0.22896012){\color[rgb]{0,0,0}\makebox(0,0)[lt]{\lineheight{1.25}\smash{\begin{tabular}[t]{l}$\zeta$\end{tabular}}}}%
    \put(0.46461196,0.10950101){\color[rgb]{0,0,0}\makebox(0,0)[lt]{\lineheight{1.25}\smash{\begin{tabular}[t]{l}$-1$\end{tabular}}}}%
    \put(0.7598297,0.11717686){\color[rgb]{0,0,0}\makebox(0,0)[lt]{\lineheight{1.25}\smash{\begin{tabular}[t]{l}$-1$\end{tabular}}}}%
    \put(0.70990662,0.34585057){\color[rgb]{0,0,0}\makebox(0,0)[lt]{\lineheight{1.25}\smash{\begin{tabular}[t]{l}$1$\end{tabular}}}}%
    \put(0.58960134,0.10908036){\color[rgb]{0,0,0}\makebox(0,0)[lt]{\lineheight{1.25}\smash{\begin{tabular}[t]{l}$\lambda$\end{tabular}}}}%
    \put(0.80419253,0.17223505){\color[rgb]{0,0,0}\makebox(0,0)[lt]{\lineheight{1.25}\smash{\begin{tabular}[t]{l}$\vec{a}$\end{tabular}}}}%
    \put(0.54823032,0.23793303){\color[rgb]{0,0,0}\makebox(0,0)[lt]{\lineheight{1.25}\smash{\begin{tabular}[t]{l}$\vec{b}$\end{tabular}}}}%
    \put(0.8246201,0.23793303){\color[rgb]{0,0,0}\makebox(0,0)[lt]{\lineheight{1.25}\smash{\begin{tabular}[t]{l}$\vec{b}$\end{tabular}}}}%
    \put(0.82084365,0.295109){\color[rgb]{0,0,0}\makebox(0,0)[lt]{\lineheight{1.25}\smash{\begin{tabular}[t]{l}$\vec{c}$\end{tabular}}}}%
    \put(0,0){\includegraphics[width=\unitlength,page=2]{Borr3Cases.pdf}}%
    \put(0.40873541,0.34585057){\color[rgb]{0,0,0}\makebox(0,0)[lt]{\lineheight{1.25}\smash{\begin{tabular}[t]{l}$\lambda'$\end{tabular}}}}%
    \put(0,0){\includegraphics[width=\unitlength,page=3]{Borr3Cases.pdf}}%
  \end{picture}%
\endgroup%

%% file: figures/GSHex.pdf_tex
%% Creator: Inkscape 1.0.2-2 (e86c870879, 2021-01-15), www.inkscape.org
%% Accompanies image file 'GSHex.pdf' (pdf, eps, ps)
\begingroup%
  \makeatletter%
  \providecommand\color[2][]{%
    \errmessage{(Inkscape) Color is used for the text in Inkscape, but the package 'color.sty' is not loaded}%
    \renewcommand\color[2][]{}%
  }%
  \providecommand\transparent[1]{%
    \errmessage{(Inkscape) Transparency is used (non-zero) for the text in Inkscape, but the package 'transparent.sty' is not loaded}%
    \renewcommand\transparent[1]{}%
  }%
  \providecommand\rotatebox[2]{#2}%
  \newcommand*\fsize{\dimexpr\f@size pt\relax}%
  \newcommand*\lineheight[1]{\fontsize{\fsize}{#1\fsize}\selectfont}%
  \ifx\svgwidth\undefined%
    \setlength{\unitlength}{240.05306454bp}%
    \ifx\svgscale\undefined%
      \relax%
    \else%
      \setlength{\unitlength}{\unitlength * \real{\svgscale}}%
    \fi%
  \else%
    \setlength{\unitlength}{\svgwidth}%
  \fi%
  \global\let\svgwidth\undefined%
  \global\let\svgscale\undefined%
  \makeatother%
  \begin{picture}(1,0.75451169)%
    \lineheight{1}%
    \setlength\tabcolsep{0pt}%
    \put(0,0){\includegraphics[width=\unitlength,page=1]{GSHex.pdf}}%
    \put(0.01030048,0.37781804){\color[rgb]{0,0,0}\makebox(0,0)[lt]{\lineheight{1.25}\smash{\begin{tabular}[t]{l}$-1$\end{tabular}}}}%
    \put(0.33734237,0.59724013){\color[rgb]{0,0,0}\makebox(0,0)[lt]{\lineheight{1.25}\smash{\begin{tabular}[t]{l}$-\zeta$\end{tabular}}}}%
    \put(0.69637327,0.29722773){\color[rgb]{0,0,0}\makebox(0,0)[lt]{\lineheight{1.25}\smash{\begin{tabular}[t]{l}$\vec{c}$\end{tabular}}}}%
    \put(0.59703502,0.21599561){\color[rgb]{0,0,0}\makebox(0,0)[lt]{\lineheight{1.25}\smash{\begin{tabular}[t]{l}$\vec{b}$\end{tabular}}}}%
    \put(0.34241805,0.21314585){\color[rgb]{0,0,0}\makebox(0,0)[lt]{\lineheight{1.25}\smash{\begin{tabular}[t]{l}$\vec{a}$\end{tabular}}}}%
    \put(0.44853232,0.44579912){\color[rgb]{0,0,0}\makebox(0,0)[lt]{\lineheight{1.25}\smash{\begin{tabular}[t]{l}$-\zeta'$\end{tabular}}}}%
    \put(0.93960899,0.37706425){\color[rgb]{0,0,0}\makebox(0,0)[lt]{\lineheight{1.25}\smash{\begin{tabular}[t]{l}$1$\end{tabular}}}}%
    \put(0.51726715,0.26522284){\color[rgb]{0,0,0}\makebox(0,0)[lt]{\lineheight{1.25}\smash{\begin{tabular}[t]{l}$\zeta'$\end{tabular}}}}%
    \put(0.61099651,0.1152559){\color[rgb]{0,0,0}\makebox(0,0)[lt]{\lineheight{1.25}\smash{\begin{tabular}[t]{l}$\zeta$\end{tabular}}}}%
    \put(0,0){\includegraphics[width=\unitlength,page=2]{GSHex.pdf}}%
  \end{picture}%
\endgroup%

%% file: figures/Case2Cases.pdf_tex
%% Creator: Inkscape 1.0.2-2 (e86c870879, 2021-01-15), www.inkscape.org
%% Accompanies image file 'Case2Cases.pdf' (pdf, eps, ps)
\begingroup%
  \makeatletter%
  \providecommand\color[2][]{%
    \errmessage{(Inkscape) Color is used for the text in Inkscape, but the package 'color.sty' is not loaded}%
    \renewcommand\color[2][]{}%
  }%
  \providecommand\transparent[1]{%
    \errmessage{(Inkscape) Transparency is used (non-zero) for the text in Inkscape, but the package 'transparent.sty' is not loaded}%
    \renewcommand\transparent[1]{}%
  }%
  \providecommand\rotatebox[2]{#2}%
  \newcommand*\fsize{\dimexpr\f@size pt\relax}%
  \newcommand*\lineheight[1]{\fontsize{\fsize}{#1\fsize}\selectfont}%
  \ifx\svgwidth\undefined%
    \setlength{\unitlength}{403.72251496bp}%
    \ifx\svgscale\undefined%
      \relax%
    \else%
      \setlength{\unitlength}{\unitlength * \real{\svgscale}}%
    \fi%
  \else%
    \setlength{\unitlength}{\svgwidth}%
  \fi%
  \global\let\svgwidth\undefined%
  \global\let\svgscale\undefined%
  \makeatother%
  \begin{picture}(1,0.32330633)%
    \lineheight{1}%
    \setlength\tabcolsep{0pt}%
    \put(0,0){\includegraphics[width=\unitlength,page=1]{Case2Cases.pdf}}%
    \put(0.15623594,0.30552749){\color[rgb]{0,0,0}\makebox(0,0)[lt]{\lineheight{1.25}\smash{\begin{tabular}[t]{l}$\zeta'$\end{tabular}}}}%
    \put(0.31690604,0.30181207){\color[rgb]{0,0,0}\makebox(0,0)[lt]{\lineheight{1.25}\smash{\begin{tabular}[t]{l}$-1$\end{tabular}}}}%
    \put(0.6262516,0.30181207){\color[rgb]{0,0,0}\makebox(0,0)[lt]{\lineheight{1.25}\smash{\begin{tabular}[t]{l}$1$\end{tabular}}}}%
    \put(0.79476213,0.30552749){\color[rgb]{0,0,0}\makebox(0,0)[lt]{\lineheight{1.25}\smash{\begin{tabular}[t]{l}$\zeta'$\end{tabular}}}}%
    \put(0.79476213,0.00188672){\color[rgb]{0,0,0}\makebox(0,0)[lt]{\lineheight{1.25}\smash{\begin{tabular}[t]{l}$\zeta$\end{tabular}}}}%
    \put(0.47508945,0.00188672){\color[rgb]{0,0,0}\makebox(0,0)[lt]{\lineheight{1.25}\smash{\begin{tabular}[t]{l}$\zeta$\end{tabular}}}}%
    \put(0.30669597,0.00188672){\color[rgb]{0,0,0}\makebox(0,0)[lt]{\lineheight{1.25}\smash{\begin{tabular}[t]{l}$1$\end{tabular}}}}%
    \put(-0.0013062,0.00188672){\color[rgb]{0,0,0}\makebox(0,0)[lt]{\lineheight{1.25}\smash{\begin{tabular}[t]{l}$-1$\end{tabular}}}}%
    \put(0.47508945,0.23399828){\color[rgb]{0,0,0}\makebox(0,0)[lt]{\lineheight{1.25}\smash{\begin{tabular}[t]{l}$\zeta'$\end{tabular}}}}%
    \put(0.15995137,0.07581487){\color[rgb]{0,0,0}\makebox(0,0)[lt]{\lineheight{1.25}\smash{\begin{tabular}[t]{l}$\zeta$\end{tabular}}}}%
    \put(0.61510532,0.15345844){\color[rgb]{0,0,0}\makebox(0,0)[lt]{\lineheight{1.25}\smash{\begin{tabular}[t]{l}$-1$\end{tabular}}}}%
    \put(0.96409161,0.15345844){\color[rgb]{0,0,0}\makebox(0,0)[lt]{\lineheight{1.25}\smash{\begin{tabular}[t]{l}$1$\end{tabular}}}}%
    \put(0.22024655,0.25547166){\color[rgb]{0,0,0}\makebox(0,0)[lt]{\lineheight{1.25}\smash{\begin{tabular}[t]{l}$C$\end{tabular}}}}%
    \put(0.53292715,0.25175391){\color[rgb]{0,0,0}\makebox(0,0)[lt]{\lineheight{1.25}\smash{\begin{tabular}[t]{l}$\vec{c}$\end{tabular}}}}%
    \put(0.87070647,0.25977687){\color[rgb]{0,0,0}\makebox(0,0)[lt]{\lineheight{1.25}\smash{\begin{tabular}[t]{l}$C$\end{tabular}}}}%
    \put(0.11492744,0.04442394){\color[rgb]{0,0,0}\makebox(0,0)[lt]{\lineheight{1.25}\smash{\begin{tabular}[t]{l}$A$\end{tabular}}}}%
    \put(0.39949379,0.2629025){\color[rgb]{0,0,0}\makebox(0,0)[lt]{\lineheight{1.25}\smash{\begin{tabular}[t]{l}$A$\end{tabular}}}}%
    \put(0,0){\includegraphics[width=\unitlength,page=2]{Case2Cases.pdf}}%
    \put(0.72121437,0.12867801){\color[rgb]{0,0,0}\makebox(0,0)[lt]{\lineheight{1.25}\smash{\begin{tabular}[t]{l}$A$\end{tabular}}}}%
    \put(0.08701671,0.0795303){\color[rgb]{0,0,0}\makebox(0,0)[lt]{\lineheight{1.25}\smash{\begin{tabular}[t]{l}$\vec{a}$\end{tabular}}}}%
    \put(0.24625454,0.16006895){\color[rgb]{0,0,0}\makebox(0,0)[lt]{\lineheight{1.25}\smash{\begin{tabular}[t]{l}$\vec{c}$\end{tabular}}}}%
    \put(0.38463559,0.14892267){\color[rgb]{0,0,0}\makebox(0,0)[lt]{\lineheight{1.25}\smash{\begin{tabular}[t]{l}$\vec{a}$\end{tabular}}}}%
    \put(0.71357744,0.07168869){\color[rgb]{0,0,0}\makebox(0,0)[lt]{\lineheight{1.25}\smash{\begin{tabular}[t]{l}$\vec{a}$\end{tabular}}}}%
    \put(0.8892835,0.22185733){\color[rgb]{0,0,0}\makebox(0,0)[lt]{\lineheight{1.25}\smash{\begin{tabular}[t]{l}$\vec{c}$\end{tabular}}}}%
    \put(0.57751227,0.23069009){\color[rgb]{0,0,0}\makebox(0,0)[lt]{\lineheight{1.25}\smash{\begin{tabular}[t]{l}$C$\end{tabular}}}}%
    \put(0,0){\includegraphics[width=\unitlength,page=3]{Case2Cases.pdf}}%
  \end{picture}%
\endgroup%

%% file: figures/Case3BR.pdf_tex
%% Creator: Inkscape 1.0.2-2 (e86c870879, 2021-01-15), www.inkscape.org
%% Accompanies image file 'Case3BR.pdf' (pdf, eps, ps)
\begingroup%
  \makeatletter%
  \providecommand\color[2][]{%
    \errmessage{(Inkscape) Color is used for the text in Inkscape, but the package 'color.sty' is not loaded}%
    \renewcommand\color[2][]{}%
  }%
  \providecommand\transparent[1]{%
    \errmessage{(Inkscape) Transparency is used (non-zero) for the text in Inkscape, but the package 'transparent.sty' is not loaded}%
    \renewcommand\transparent[1]{}%
  }%
  \providecommand\rotatebox[2]{#2}%
  \newcommand*\fsize{\dimexpr\f@size pt\relax}%
  \newcommand*\lineheight[1]{\fontsize{\fsize}{#1\fsize}\selectfont}%
  \ifx\svgwidth\undefined%
    \setlength{\unitlength}{144.0581951bp}%
    \ifx\svgscale\undefined%
      \relax%
    \else%
      \setlength{\unitlength}{\unitlength * \real{\svgscale}}%
    \fi%
  \else%
    \setlength{\unitlength}{\svgwidth}%
  \fi%
  \global\let\svgwidth\undefined%
  \global\let\svgscale\undefined%
  \makeatother%
  \begin{picture}(1,0.90492676)%
    \lineheight{1}%
    \setlength\tabcolsep{0pt}%
    \put(0,0){\includegraphics[width=\unitlength,page=1]{Case3BR.pdf}}%
    \put(0.64586037,0.70325681){\color[rgb]{0,0,0}\makebox(0,0)[lt]{\lineheight{1.25}\smash{\begin{tabular}[t]{l}$C$\end{tabular}}}}%
    \put(0.44794168,0.00528759){\color[rgb]{0,0,0}\makebox(0,0)[lt]{\lineheight{1.25}\smash{\begin{tabular}[t]{l}$\zeta$\end{tabular}}}}%
    \put(0.7153041,0.44047917){\color[rgb]{0,0,0}\makebox(0,0)[lt]{\lineheight{1.25}\smash{\begin{tabular}[t]{l}$\vec{c}$\end{tabular}}}}%
    \put(-0.00366062,0.00528759){\color[rgb]{0,0,0}\makebox(0,0)[lt]{\lineheight{1.25}\smash{\begin{tabular}[t]{l}$-1$\end{tabular}}}}%
    \put(0.8993668,0.00528759){\color[rgb]{0,0,0}\makebox(0,0)[lt]{\lineheight{1.25}\smash{\begin{tabular}[t]{l}$1$\end{tabular}}}}%
    \put(0.44794168,0.85510162){\color[rgb]{0,0,0}\makebox(0,0)[lt]{\lineheight{1.25}\smash{\begin{tabular}[t]{l}$\zeta'$\end{tabular}}}}%
    \put(0.63200435,0.1094188){\color[rgb]{0,0,0}\makebox(0,0)[lt]{\lineheight{1.25}\smash{\begin{tabular}[t]{l}$\vec{a}$\end{tabular}}}}%
    \put(0,0){\includegraphics[width=\unitlength,page=2]{Case3BR.pdf}}%
    \put(0.24886527,0.1094188){\color[rgb]{0,0,0}\makebox(0,0)[lt]{\lineheight{1.25}\smash{\begin{tabular}[t]{l}$\vec{a}$\end{tabular}}}}%
    \put(0,0){\includegraphics[width=\unitlength,page=3]{Case3BR.pdf}}%
  \end{picture}%
\endgroup%